\documentclass[12pt,a4paper]{amsart}
\usepackage[top=1.14in, bottom=1.14in, left=1.16in, right=1.16in]{geometry}
\usepackage[colorlinks=true, pdfstartview=FitV, linkcolor=blue, citecolor=blue, urlcolor=blue, breaklinks=true]{hyperref}

\usepackage[english]{babel}
\usepackage[T1]{fontenc}
\usepackage{latexsym}
\usepackage{amsmath}
\usepackage{amssymb, amscd, amsthm, mathrsfs}
\usepackage[all]{xy}
\usepackage[dvips]{graphicx}

\numberwithin{equation}{section}

\newtheorem{definition}{Definition}
\newtheorem{proposition}{Proposition}
\newtheorem{theorem}{Theorem}
\newtheorem{lemma}{Lemma}

\newtheorem{corollary}{Corollary}
\newtheorem{remark}{Remark}
\newtheorem{example}{Example}

\newtheorem*{thma}{Theorem A}
\newtheorem*{thmb}{Theorem B}
\newtheorem*{thmc}{Theorem C}

\newcommand{\ra}{\rightarrow}

\newcommand{\mc}{\mathbb{C}}

\newcommand{\SL}{\operatorname*{SL}}

\newcommand{\Gr}{{\operatorname*{Gr}}}

\newcommand{\sspan}{\operatorname*{span}}
\newcommand{\id}{\operatorname*{id}}

\newcommand{\pr}{{\operatorname*{pr}}}

\newcommand{\msl}{\mathfrak{sl}}
\newcommand{\cO}{\mathcal O}
\newcommand{\bP}{{\mathbb P}}
\newcommand{\bd}{{\mathbf d}}
\newcommand{\be}{{\mathbf e}}
\newcommand{\bC}{{\mathbb C}}
\newcommand{\fa}{{\mathfrak a}}
\newcommand{\hk}{\hookrightarrow}
\newcommand{\bs}{{\mathbf s}}
\newcommand{\bG}{{\mathbb G}}
\newcommand{\la}{\lambda}
\newcommand{\om}{\omega}
\newcommand{\bE}{{\mathbf E}}
\newcommand{\fb}{{\mathfrak b}}
\newcommand{\fn}{{\mathfrak n}}
\newcommand{\fg}{{\mathfrak g}}
\newcommand{\bZ}{{\mathbb Z}}
\newcommand{\T}{\otimes}

\newcommand{\ZZ}{\mathbb{Z}}

\begin{document}

\title[Linear degenerations of flag varieties]
{Linear degenerations of flag varieties: partial flags, defining equations, and group actions}

\author{G. Cerulli Irelli}
\address{Giovanni Cerulli Irelli:\newline
Dipartimento  S.B.A.I.,  Sapienza Universit\'a di Roma, Via Scarpa 10, 00161, Roma, Italy}
\email{giovanni.cerulliirelli@uniroma1.it} 
\author{X. Fang} 
\address{Xin Fang:\newline
University of Cologne, Mathematical Institute, Weyertal 86--90, 50931 Cologne, Germany}
\email{xinfang.math@gmail.com}
\author{E. Feigin} 
\address{Evgeny Feigin:\newline National Research University Higher School of Economics, Department of Mathematics, 
Usacheva str. 6, 119048, Moscow, Russia, {\it and } 
Skolkovo Institute of Science and Technology, Skolkovo Innovation Center, Building 3,
Moscow 143026, Russia}
\email{evgfeig@gmail.com}
\author{G. Fourier} 
\address{Ghislain Fourier:\newline
RWTH Aachen University, Pontdriesch 10--16, 52062 Aachen}
\email{fourier@mathb.rwth-aachen.de}
\author{M. Reineke}
\address{Markus Reineke:\newline
Ruhr-Universit\"at Bochum, Faculty of Mathematics, Universit\"atsstra{\ss}e 150, 44780 Bochum, Germany}
\email{markus.reineke@rub.de}

\begin{abstract}

We continue, generalize and expand our study of linear degenerations of flag varieties from \cite{CFFFR}. We realize partial flag varieties as quiver Grassmannians for equi-oriented type A quivers and construct linear degenerations by varying the corresponding quiver representation.  We prove that there exists the deepest flat degeneration and the deepest flat irreducible degeneration: the former 
is the partial analogue of the mf-degenerate flag variety and the latter coincides with the partial PBW-degenerate
flag variety. We compute the generating function of the number of orbits in the flat irreducible locus
and study the natural family of line bundles on the degenerations from the flat irreducible locus. We also describe 
explicitly the reduced scheme structure on these degenerations and conjecture that similar results hold for the whole flat locus.
Finally, we prove an analogue of the Borel-Weil theorem for the flat irreducible locus.      
\end{abstract}
\maketitle

\section{Introduction}
The theory of complex simple Lie groups and Lie algebras is known to be closely related to the representation theory of Dynkin quivers
(see e.g. \cite{AdF,CFR2,G,R_Mon}). In this paper we use the following simple but powerful observation: any partial flag variety 
associated to the group $\mathrm{SL}_N$ is isomorphic to a quiver Grassmannian for the equi-oriented type $A$ quiver and suitably chosen 
representation and dimension vector. Varying the representation of the quiver and keeping the dimension vector fixed
one gets degenerations of the flag varieties (see e.g. \cite{CFFFR,CFR1,CFRSchubert,CL}). The goal of this paper is to study 
these degenerations, in particular, to describe the irreducible and flat irreducible loci. Let us formulate
the setup and our results in more details.
 
Let $G=\mathrm{SL}_N(\mathbb{C})$ and let $P$ be a parabolic subgroup of $G$ with respect to the fixed Borel subgroup $B$.
The quotient $G/P$ is known to be isomorphic to the variety of flags $(U_1\subset U_2\subset\dots\subset U_n)$
in an $N$-dimensional vector space such that $\dim U_i=e_i$ for a certain increasing sequence $1\le e_1<\dots <e_n\leq N$.  

Let $Q$ be the equi-oriented quiver of type $A_n$ with the set of vertices $Q_0=\{1,2,\cdots,n\}$ where $n$ is the sink. We fix $N\geq n+1$ and a complex vector space $V$ of dimension $N$. 
We consider the dimension vector $\mathbf{d}=(N,\dots,N)$ and denote by $R_\mathbf{d}$ the affine space whose points parametrize the 
$Q$--representations of dimension vector $\mathbf{d}$, i.e. collections $\{(f_i)_{i=1}^{n-1}\}$ of linear endomorphisms of $V$. 
The group $G_\mathbf{d}=\prod_{i=1}^n \mathrm{GL}_N$ acts on $R_\mathbf{d}$ by base change and the $G_\mathbf{d}$-orbits get identified with 
the isomorphism classes of quiver representations. It is known that there are only finitely many orbits, parametrized by the 
collections $(r_{i,j})_{1\le i<j\le n}$ of the ranks
of the composite maps. A general point of $R_\mathbf{d}$ is isomorphic to $M^0:=(\id_V,\cdots, \id_V)$.  For a point 
$M=(f_i)_{i=1}^{n-1}\in R_\mathbf{d}$ we denote by $\mathbf{r}^M=(r_{i,j}^M)$ the rank collection $r_{i,j}^M=\mathrm{rank}(f_{j-1}\circ\cdots\circ f_i)$. In particular, if $M=M^0$, then $r_{i,j}^{M}=N$ 
for all pairs $i,j$ and we denote this collection by $\mathbf{r}^0$.  

We fix a dimension vector $\mathbf{e}=(e_i)_{i=1}^n$ such that $1\le e_1<\dots <e_n\leq N$ and consider 
the proper family $\pi:Y_\mathbf{e}\rightarrow R_\mathbf{d}$ whose fiber over a point $M$ is the quiver Grassmannian 
$\Gr_\mathbf{e}(M)$.  Our goal is to study geometric properties of this family.

Two simple observations are in order. The first observation is that a general fibre of this family is isomorphic to 
$G/P$, thus the special fibres can be viewed as degenerations of the partial flag varieties. 
The second observation is as follows. The map $\pi$ is $G_\mathbf{d}$-equivariant and the quiver Grassmannians corresponding to the points from one $G_\mathbf{d}$-orbit
are isomorphic.  We denote by $\mathcal{O}_{{\mathbf r}}$ the $G_\mathbf{d}$-orbit corresponding to the tuple ${\mathbf r}$. 
The main message of our paper is that there exist two other rank collections $\mathbf{r}^1$ and $\mathbf{r}^2$:
\begin{gather}\label{Eq:r1}
r^1_{i,j}=N-e_j+e_i,\ 1\le i<j\le n;\\\label{Eq:r2}
r^2_{i,j}=N-1-e_j+e_i,\ 1\le i<j\le n, 
\end{gather}
which are as fundamental as the tuple $\mathbf{r}^0$. In particular, the rank collection $\mathbf{r}^1$ corresponds to
the PBW degenerate flag variety \cite{Feigin1,FFFM,Fou1}.  We provide here some details. 

The partial flag varieties $G/P$ are known to be irreducible and have easily computed dimensions.
There are two natural loci in $R_\mathbf{d}$. The first one is the flat locus $U_{flat}$ which is the locus where the map 
$\pi$ is flat. In other words, $U_{flat}$ consists of representations $M$ such that the quiver Grassmannian  $\Gr_\mathbf{e}(M)$
is of expected (minimal possible) dimension $\dim G/P$. The second natural locus is the flat irreducible locus 
$U_{flat,irr}\subset U_{flat}$ consisting of $M$ such that $\Gr_\mathbf{e}(M)$ is irreducible.   
Here is our first theorem which generalizes \cite[Theorem~3]{CFFFR}.

\begin{thma}\label{Thm:A} The following holds:
\begin{itemize}\item[a)] The flat irreducible locus $U_{flat,irr}$ consists of the orbits  $\mathcal{O}_{{\mathbf r}}$ 
degenerating to $\mathcal{O}_{{\mathbf r^1}}$, i.e. $r_{i,j}\ge r^1_{i,j}$ for all pairs $i,j$.
\item[b)] The flat locus $U_{flat}$ consists of the orbits  $\mathcal{O}_{{\mathbf r}}$ 
degenerating to $\mathcal{O}_{{\mathbf r^2}}$, i.e. $r_{i,j}\ge r^2_{i,j}$ for all pairs $i,j$.
\end{itemize}
\end{thma}

Our next goal is to compute the number of orbits in the flat irreducible locus.  Let $B_{\mathbf e}$ be the number of these orbits. We note that $B_{\mathbf e}$ does not depend on $N$ (provided $N>e_n$).  
If $e_i=i$, then  $B_{\mathbf e}$ is equal to the $n$-th Bell number https://oeis.org/A000110 
(see \cite[Section 4.2]{CFFFR}).

We consider the generating function 
$$
B_n(x_1,\dots,x_n)=\sum_{{\mathbf e}} B_{\mathbf e}\, x_1^{e_1}x_2^{e_2-e_1}\dots x_n^{e_n-e_{n-1}}.
$$
\begin{thmb} We have
\[
B_n(x_1,\dots,x_n)=\prod_{i=1}^n (1-x_i)^{-1}\prod_{\emptyset\ne I\subset\{2,\dots,n\}} (1-\prod_{i\in I} x_i)^{-1}.
\]
\end{thmb}

Next, we describe the reduced scheme structure for the quiver Grassmannians corresponding to the representations
in $U_{flat, irr}$ by providing an explicit set of quadratic generators for the ideal describing the Pl\"ucker embedding
(see also \cite{LW}).
Our main combinatorial tool is the notion of PBW semi-standard Young
tableaux (see \cite{Feigin2}), parametrizing a basis in the homogeneous coordinate ring of the PBW degenerate flag
varieties. We prove the following theorem.

\begin{thmc}
For any orbit $\mathcal{O}$ degenerating to $\mathcal{O}_{{\mathbf r}^1}$ there exists a point $M\in \mathcal{O}$ such
that the semi-standard PBW tableaux form a basis in the homogeneous coordinate ring of ${\rm Gr}_{\mathbf e}(M)$.
\end{thmc}
We conjecture that a similar result holds for the whole flat locus. 

Finally, we discuss groups acting on the fibers in the flat irreducible locus and study the sections of 
natural line bundles. More precisely, we make use of a transversal slice $T$ through the flat irreducible locus constructed in \cite{CFFFR}. For a $Q$-representation 
$M_t$ for $t\in T$ we construct a group $G_t$ acting on the quiver Grassmannian  $\Gr_\be(M_t)$ with an open dense orbit.
We construct a family of representations of $G_t$ and identify them with the dual spaces of sections of natural line bundles 
on $\Gr_\be(M_t)$.

Our paper is organized as follows. In section~\ref{quiver} we recall some basic facts about quivers and quiver Grassmannians of type $A$. In section~\ref{Sec:ProofThmA} we prove Theorem~A. In section~\ref{Sec:ProofThmB} we prove Theorem~B. In section~\ref{Sec:Plucker} we describe the ideal of relations defining linear flat degenerations and prove Theorem~C. In section~\ref{Sec:LineBundles} we construct line bundles on the flat degenerations of the complete flag variety and 
provide a Borel-Weil-type theorem for quiver Grassmannians. 

\vskip 5pt
\noindent
\textbf{Acknowledgments}. The work of the authors is supported by the DFG-RSF project ``Geometry and representation theory at the interface between Lie algebras and quivers''. E.F. was partially supported by the Russian Academic Excellence Project '5-100'.

\section{Methods from the representation theory of quivers}\label{quiver}

\subsection{Quiver representations}

For all basic definitions and facts on the representation theory of (Dynkin) quivers, we refer to \cite{ASS}.

Let $Q$ be a finite quiver with the set of vertices $Q_0$ and arrows written $a:i\rightarrow j$ for $i,j\in Q_0$. 
We  assume that $Q$ is a Dynkin quiver, that is, its underlying unoriented graph $|Q|$ is a disjoint union of simply-laced Dynkin diagrams. 
\par
We consider (finite-dimensional) $\mathbb{C}$-representations of $Q$. Such a representation is given by a tuple 
$$M=((M_i)_{i\in Q_0},(f_a)_{a:i\rightarrow j}),$$
where $M_i$ is a finite-dimensional $\mathbb{C}$-vector space for every vertex $i$ of $Q$, and $f_a:M_i\rightarrow M_j$ 
is a $\mathbb{C}$-linear map for every arrow $a:i\rightarrow j$ in $Q$. A morphism between representations $M$ and 
$K=((K_i)_i,(g_a)_a)$ is a tuple of $\mathbb{C}$-linear maps $(\varphi_i:M_i\rightarrow K_i)_{i\in Q_0}$ such that 
$\varphi_jf_a=g_a\varphi_i$ for all $a:i\rightarrow j$ in $Q$. Composition of morphisms is defined componentwise, resulting in a $\mathbb{C}$-linear category ${\rm rep}_\mathbb{C}Q$. This category is $\mathbb{C}$-linearly equivalent to the category $\bmod A$ of finite-dimensional left modules over the path algebra $A=\mathbb{C} Q$ of $Q$.
\par
For a vertex $i\in Q_0$, we denote by $S_i$ the simple representation associated to $i$, namely, $(S_i)_i=\mathbb{C}$ and 
$(S_i)_j=0$ for all $j\not=i$, and all maps being identically zero; every simple representation is of this form. We let 
$P_i$ be a projective cover of $S_i$, and $I_i$ an injective hull of $S_i$. 
\par
The Grothendieck group $K_0({\rm rep}_\mathbb{C}Q)$ is isomorphic to the free abelian group ${\mathbb Z}Q_0$ in $Q_0$ 
via the map attaching to the class of a representation $M$ its dimension vector 
${\rm\mathbf{dim}\,} M=(\dim M_i)_{i\in Q_0}\in\ZZ Q_0$. The category ${\rm rep}_\mathbb{C}Q$ is hereditary, that is, ${\rm Ext}^{\geq 2}(\_,\_)$ vanishes identically, and its homological Euler form
$$\dim {\rm Hom}(M,K)-\dim {\rm Ext}^1(M,K)=\langle{\rm\mathbf{dim}\,} M,{\rm\mathbf{dim}\,} K\rangle$$
is given by
$$\langle{\mathbf d},{\mathbf e}\rangle=\sum_{i\in Q_0}d_ie_i-\sum_{a:i\rightarrow j}d_ie_j.$$
For two dimension vectors $\mathbf{e},\mathbf{d}\in\mathbb{N}{Q_0}$ we write $\mathbf{e}\leq\mathbf{d}$ if $e_i\leq d_i$ for all $i\in Q_0$. 

By Gabriel's theorem, the isomorphism classes $[U_\alpha]$ of indecomposable representations $U_\alpha$ of $Q$ correspond bijectively to the positive roots $\alpha$ of the root system 
$\Phi$ of type $|Q|$; more concretely, we realize $\Phi$ as the set of vectors $\alpha\in{\mathbb Z}Q_0$ such that $\langle\alpha,\alpha\rangle=1$; then there exists a unique (up to isomorphism) indecomposable representation $U_\alpha$ such that ${\rm\mathbf{dim}\,} U_\alpha=\alpha$ for every $\alpha\in\Phi^+=\Phi\cap{\mathbb N}Q_0$.

We make our discussion of the representation theory of a Dynkin quiver so far explicit in the case of the equi-oriented type 
$A_n$ quiver $Q$ given as
$$
\xymatrix{
1\ar[r]&2\ar[r]&\ar[r]\cdots\ar[r]&n.
}
$$We identify ${\mathbb Z}Q_0$ with ${\mathbb Z}^n$, and the Euler form is then given by
$$\langle{\mathbf d},{\mathbf e}\rangle=\sum_{i=1}^nd_ie_i-\sum_{i=1}^{n-1}d_ie_{i+1}.$$
We denote the indecomposable representations by $U_{i,j}$ for $1\leq i\leq j\leq n$, where $U_{i,j}$ is given as
$$0\rightarrow\ldots\rightarrow 0\rightarrow \mathbb{C}\stackrel{{\rm id}}{\rightarrow}\ldots\stackrel{{\rm id}}{\rightarrow}\mathbb{C}\rightarrow 0\rightarrow\ldots\rightarrow 0,$$
supported on the vertices $i,\ldots,j$. In particular, we have $S_i=U_{i,i}$, $P_i=U_{i,n}$, $I_i=U_{1,i}$ for all $i$.

We have
$$\dim{\rm Hom}(U_{i,j},U_{k,l})=1\mbox{ if and only if }k\leq i\leq l\leq j$$
and zero otherwise, and we have
$$\dim{\rm Ext}^1(U_{k,l},U_{i,j})=1\mbox{ if and only if }k+1\leq i\leq l+1\leq j,$$
and zero otherwise, where the extension group, in case it is non-zero, is generated by the class of the exact sequence
$$0\rightarrow U_{i,j}\rightarrow U_{i,l}\oplus U_{k,j}\rightarrow U_{k,l}\rightarrow 0,$$ 
where we formally set $U_{i,j}=0$ if $i<1$ or $j>n$ or $j<i$.
\par
Given two dimension vectors ${\mathbf e}$ and ${\mathbf s}$ such that $e_0:=0\leq e_1\leq e_2\leq\ldots\leq e_n$ and 
$s_1\geq s_2\geq\ldots\geq s_n\geq s_{n+1}:=0$, we define the two $Q$--representations: 
\begin{equation}\label{Eq:PeIf}
P^{\mathbf e}:=\bigoplus_{i=1}^nP_i^{e_i-e_{i-1}},\ \ I^{\mathbf s}:=\bigoplus_{i=1}^nI_i^{s_i-s_{i+1}}.
\end{equation}

Given a dimension vector ${\mathbf d}\in{\mathbb N}Q_0$ and $\mathbb{C}$-vector spaces $V_i$ of dimension $d_i$ ($i\in Q_0$), let $R_{\mathbf d}$ be the affine space $$R_{\mathbf d}=\bigoplus_{i=1}^{n-1}{\rm Hom}_\mathbb{C}(V_i,V_{i+1}),$$
on which the group
$G_{\mathbf d}={\rm GL}(V_1)\times \cdots\times  {\rm GL}(V_n)$
acts via base change: given $g=(g_i)_{i=1}^n\in G_\mathbf{d}$ and $f=(f_i)_{i=1}^{n-1}\in R_\mathbf{d}$, we have $g\cdot f=f'$ where $f'$ makes commutative every square
$$
\xymatrix{V_i\ar^{f_i}[r]\ar_{g_i}[d]&V_{i+1}\ar^{g_{i+1}}[d]\\V_i\ar^{f'_i}[r]&V_{i+1}}
$$
for $i\in Q_0$. The $G_{\mathbf d}$-orbits in $R_{\mathbf d}$ are naturally parametrized by isomorphism classes of representations of $Q$ of dimension vector ${\mathbf d}$. 
By the Krull-Schmidt theorem, a $Q$-representation $M$ is, up to isomorphism, determined by the multiplicities of the $U_{i,j}$, that is,
$$M=\bigoplus_{i\leq j}U_{i,j}^{m_{i,j}}.$$
Then ${\rm\mathbf{dim}\,} M={\mathbf d}$ is equivalent to
$$\sum_{k\leq i\leq l}m_{k,l}=d_i\mbox{ for all }i.$$
We define
$$r_{i,j}(M)=\sum_{k\leq i\leq j\leq l}m_{k,l}$$
for $i\leq j$. We note that $r_{i,j}$ is equal to the rank of the composite map $M_i\to M_j$.
Viewing $M$ as a tuple of maps $(f_1,\ldots,f_{n-1})$ as before, $r_{i,j}$ is thus the rank of $f_{j-1}\circ\ldots\circ f_i$ and, trivially, we have $r_{i,i}=n+1$. We can recover $m_{i,j}$ from $(r_{k,l})_{k,l}$ via
$$m_{i,j}=r_{i,j}-r_{i,j+1}-r_{i-1,j}+r_{i-1,j+1},$$
for all $1\leq i\leq j\leq n$, where we formally set $r_{i,j}=0$ if $i=0$ or $j=n+1$ and $r_{i,i}=n+1$. 
We easily derive the inequality
\begin{equation}\label{recin}r_{i,l}+r_{j,k}\geq r_{i,k}+r_{j,l}\end{equation}
for all four-tuples $i<j\leq k<l$.

Let $\mathcal{O}_{\mathbf r}$ be a subset of $R_{\mathbf d}$ consisting of maps $(f_1,\dots,f_{n-1})$ such that
\[
{\rm rank} (f_{j-1}\circ\dots\circ f_i)=r_{i,j}.
\]
If non-empty, $\mathcal{O}_{\mathbf r}$ is a single $G_\mathbf{d}$-orbit, and every orbit arises in this way.

The orbit of $M$ degenerates to the orbit of $K$ if $K$ (or $\mathcal{O}_K$) is contained in the closure of $\mathcal{O}_M$. 
In this case we write $M\leq_{deg} K$. By \cite{Bo}, we have for any $U$
\begin{equation}\label{Eq:DegChar}
M\leq_{deg} K\textrm{ if and only if }\dim{\rm Hom}(U,M)\leq\dim{\rm Hom}(U,K).
\end{equation}

\subsection{Dimension estimates for certain quiver Grassmannians}\label{Sec:QG}

Let $Q$ be an equi-oriented quiver of type $A_n$. Let $N\geq n+1$ and let $V$ be a complex vector space of dimension $N$. 
Given the dimension vector $\mathbf{d}=(N,\cdots, N)\in\mathbb{N}^{n}$, the variety $R_\mathbf{d}$ consists of collections 
$(f_i:V\rightarrow V)_{i=1}^{n-1}$ of linear endomorphisms of $V$. Let ${\mathbf e}=(e_1,\cdots, e_n)$ be a dimension vector such that 
$e_0:=0<1\leq e_1\leq\cdots\leq e_n\leq e_{n+1}:=N$, 
$Z_{\mathbf e}={\rm Gr}_{e_1}(V)\times\ldots\times{\rm Gr}_{e_n}(V)$ and let $Y_{\mathbf e}\subset R_{\mathbf d}\times Z_{\mathbf e}$ be the variety of compatible pairs of
sequences $(f_*,U_*)$ such that $f_i(U_i)\subset U_{i+1}$ for all $i$. The natural projection $\pi:Y_{\mathbf{e}}\to R_{\mathbf d}$ is called the universal quiver Grassmannian and it is the family mentioned in the introduction  that we want to study. 
It is $G_\mathbf{d}$-equivariant and the quiver Grassmannian for a $Q$-representation $M\in R_\mathbf{d}$ is defined as 
${\rm Gr}_{\mathbf e}(M)=\pi^{-1}(M)$.

We would like to estimate the dimension of ${\rm Gr}_{\mathbf e}(M)$. A general representation $M^0$ of dimension vector ${\mathbf d}$ is isomorphic to 
$U_{1,n}^{N}$, thus all its arrows are represented by the identity maps.  Since ${\rm Gr}_{\mathbf e}(M^0)$ is a partial $\SL_{N}$-flag variety, we denote by ${\rm Fl}^{\mathbf r}(V)$ the $\pi$-fiber over a point in $\mathcal{O}_{\mathbf r}$, which is well-defined up to isomorphism since $\pi$ is $G_\mathbf{d}$-equivariant. We call ${\rm Fl}^{\mathbf r}(V)$ the ${\mathbf r}$-degenerate partial flag variety.

It follows from \cite[Prop.~2.2]{CFR1} that every irreducible component of ${\rm Fl}^{\mathbf r}(V)$ has dimension at least 
$$
\textrm{dim}({\rm Fl}^{\mathbf r}(V))\geq \langle\mathbf{e},\mathbf{d-e}\rangle=\sum_{i=1}^ne_i(e_{i+1}-e_i)=\textrm{dim}(\mathrm{SL}_N/P)
$$
where $P$ is an appropriate parabolic subgroup. 
We would like to study for which rank collections $\mathbf r$ this dimension estimate is an equality, and in case the equality holds, 
how many irreducible components the corresponding ${\mathbf r}$-degenerate partial flag varieties  have. It turns out that this can 
be done by a straightforward modification of the proof of \cite[Theorem~1, Proposition~1]{CFFFR}. We get the following complete answer. 

To state the result we need to recall the stratification of ${\rm Gr}_{\mathbf e}(M)$ introduced in \cite{CFR1}. Namely, for 
a representation $K$ of dimension vector ${\mathbf e}$, let $\mathcal{S}_{[K]}$ be the subset of ${\rm Gr}_{\mathbf e}(M)$ consisting 
of all sub-representations $U\subset M$ which are isomorphic to $K$. Then $\mathcal{S}_{[K]}$ is known to be an irreducible locally 
closed subset of ${\rm Gr}_{\mathbf e}(M)$ of dimension $\dim{\rm Hom}(K,M)-\dim{\rm End}(K)$. Since this gives a stratification of 
${\rm Gr}_{\mathbf e}(M)$ into finitely many irreducible locally closed subsets, the irreducible components of ${\rm Gr}_{\mathbf e}(M)$ 
are necessarily of the form $\overline{\mathcal{S}_{[K]}}$ for certain $K$.

\begin{theorem}\label{tc} Let $Q$ be the equi-oriented quiver of type $A_n$. Let ${\mathbf d}=(N,\cdots, N)$, 
$\mathbf{e}=(e_1\leq\cdots\leq e_n)$ and $\mathbf{f}=\mathbf{d-e}$ be dimension vectors as above.  
Let $M$ be a $Q$--representation of dimension vector ${\mathbf d}$, written as $M=P\oplus X$, where $P$ is  projective.
\begin{enumerate}
\item The quiver Grassmannian ${\rm Gr}_{\mathbf e}(M)$ has dimension $\langle\mathbf{e},\mathbf{d-e}\rangle$ if and only if, for all subrepresentations $\overline{K}$ of $X$ such that ${\mathbf e}-{\rm\mathbf{dim}\,}\overline{K}\leq {\rm \mathbf{dim}}\,P$, we have
$$\dim{\rm End}(\overline{K})\geq \dim{\rm Hom}(\overline{K},X)-\dim{\rm Hom}(\overline{K},I^{\mathbf f}).$$
\item In this case, the irreducible components of ${\rm Gr}_{\mathbf e}(M)$ are of the form $\overline{\mathcal{S}_{[K]}}$ 
for representations $K=K_P\oplus\overline{K}$ such that, $K_P$ is projective, $\overline{K}$ has no projective direct summands 
and in the previous inequality for $\overline{K}$, equality holds.
\end{enumerate}
\end{theorem}

\begin{proof}
This is a straightforward modification of the proof of \cite[Theorem~1]{CFFFR}.
\end{proof}

\section{Flat and flat-irreducible locus}\label{Sec:ProofThmA}
In this section we prove Theorem~A of the introduction. 
\subsection{Complements of certain open loci in \texorpdfstring{$R_{\mathbf d}$}{the base space}}

We retain the notation of the previous section. Thus, $Q$ is the equi-oriented quiver of type $A_n$, $N\geq n+1$, $\mathbf{d}=(N,\cdots, N)\in\mathbb{N}^n$, $\mathbf{e}=(e_1\leq\cdots\leq e_n)\leq\mathbf{d}$ and $\mathbf{f}=\mathbf{d-e}$. We are going to show the technical key result to prove Theorem~\ref{Thm:A}. 
We introduce some special representations in $R_\mathbf{d}$: for a tuple ${\mathbf a}=(a_1,\ldots,a_{n-1})$ of non-negative integers $a_i$ such that $\sum_{i<n}a_i\leq N$, we define $M({\mathbf a})$ by the multiplicities:
$$m_{1,n}=N-\sum_ia_i,\; m_{1,i}=a_i\mbox{ for }i<n,\; m_{i,n}=a_{i-1}\mbox{ for }i>1,$$
and $m_{j,k}=0\mbox{ for all other }j<k.$
In particular, we define
$$M^0=M(0,\ldots,0),\;
M^1=M(e_2-e_1,\ldots,e_n-e_{n-1}).$$

It is easily verified that
$$r_{ij}(M({\mathbf a}))=N-\sum_{i\leq k<j}a_k.$$

We also define $M^2$ by the multiplicities $$m_{1,1}=e_2-e_1+1,\; m_{n,n}=e_n-e_{n-1}+1,\; m_{1,i}=e_{i+1}-e_i\mbox{ for all }i>1,$$
$$m_{i,n}=e_i-e_{i-1}\mbox{ for all }i<n,\; m_{i,i}=1\mbox{ for all }1<i<n,$$
and $m_{j,k}=0$ for all other $j<k$.

A direct calculation then shows that
$${\mathbf r}(M^0)={\mathbf r}^0,\;
{\mathbf r}(M^1)={\mathbf r}^1,\;
{\mathbf r}(M^2)={\mathbf r}^2$$
as defined in \eqref{Eq:r1} and \eqref{Eq:r2}, respectively.
In more invariant terms, we can write $M^1=P^{\mathbf e}\oplus I^{\mathbf f}$. There exists a short exact sequence
$$0\rightarrow P^{\mathbf e}\rightarrow M^0\rightarrow I^{\mathbf f}\rightarrow 0.$$
We have canonical maps
$$P^{\mathbf e}\rightarrow S=\bigoplus_{i=1}^n S_i\rightarrow I^{\mathbf f},$$
and $M^2$ can be written as
\begin{equation}\label{Eq:DefM2}
M^2\simeq P^{\mathbf e}\oplus S\oplus (I^{\mathbf f}/S).
\end{equation}

Now we turn to degenerations of representations. Again we write $M\leq_{deg} K$ if the closure of the $G_{\mathbf d}$-orbit of $M$ 
contains $K$; the numerical characterization \eqref{Eq:DegChar} of degenerations mentioned above then reads
$$M\leq_{deg} K\mbox{ if and only if } r_{i,j}(M)\geq r_{i,j}(K)\mbox{ for all }i<j.$$
The representation $M^0=U_{1,n}^{N}$ is generic in the sense that $M^0\leq M$ for all $M$ in $R_\mathbf{d}$.
The following result characterizes representations $M\in R_\mathbf{d}$ that degenerate to $M^1$.
\begin{proposition}
Given $M\in R_\mathbf{d}$ we have: $M\leq_{deg} M^1$ if and only if there exists a short exact sequence $0\rightarrow P^{\mathbf e}\rightarrow M\rightarrow I^{\mathbf f}\rightarrow 0$.
\end{proposition}
\begin{proof}
If $M$ fits into the stated exact sequence then $M$ degenerates to $P^\mathbf{e}\oplus I^\mathbf{f}$ (\cite[Lemma~1.1]{Bo2}). On the other hand, suppose that $M\leq_{deg} M^1=P^\mathbf{e}\oplus I^\mathbf{f}$. Since $[P^\mathbf{e},M]=[P^\mathbf{e},M^1]$ we can conclude that $P^\mathbf{e}$ embeds into $M$ by \cite[Theorem~2.4]{Bo2} and the generic quotient of $M$ by $P^\mathbf{e}$ is $I^\mathbf{f}$. 
\end{proof}

We are now interested in the complement of the locus of representations degenerating into $M^1$ resp. $M^2$. For this, we introduce the following tuples:
\begin{itemize}
\item for $1\leq i<n$, define $${\mathbf a}^i=(0,\ldots,0,e_{i+1}-e_i+1,0,\ldots,0),$$ with the $i$-th entry being non-zero;
\item  for $1\leq i<j<n$, define 
$${\mathbf a}^{i,j}=(0,\ldots,0,e_{i+1}-e_i+1,e_{i+2}-e_{i+1},\ldots,e_{j-1}-e_{j-2},e_j-e_{j-1}+1,0,\ldots,0),$$
with the non-zero entries placed between the $i$-th and the $(j-1)$-st entry, except in the case $j=i+1$, where we define   
$${\mathbf a}^{i,i+1}=(0,\ldots,0,e_{i+1}-e_i+2,0,\ldots,0),$$ with the $i$-th entry being non-zero.
\end{itemize}

Now we can formulate:

\begin{theorem}\label{t5} Let $M$ be a representation in $R_{\mathbf d}$.
\begin{enumerate}
\item If $M$ degenerates to $M^2$ but not to $M^1$, then $M$ is a degeneration of $M({\mathbf a}^i)$ for some $i$.
\item If $M$ does not degenerate to $M^2$, then $M$ is a degeneration of $M({\mathbf a}^{i,j})$ for some $i<j$.
\end{enumerate}
\end{theorem}

\begin{proof} To prove the first part, let $M$ degenerate to $M^2$ but not to $M^1$ and consider the corresponding rank collection ${\mathbf r}={\mathbf r}(M)$. 
Degeneration of $M$ to $M^2$ is equivalent to ${\mathbf r}\geq{\mathbf r}^2$ componentwise, thus $r_{i,j}\geq N-1-e_j+e_i$ for all $i<j$. Non-degeneration of $M$ to $M^1$ is equivalent to ${\mathbf r}\not\geq{\mathbf r}^1$, thus there exists a pair $i<j$ such that $r_{i,j}<N-e_j+e_i$, which implies $r_{i,j}=N-1-e_j+e_i$. We claim that this equality already holds for a pair $i<j$ such that $j=i+1$. Suppose, to the contrary, that $r_{i,j}=N-1-e_j+e_i$ for some pair $i<j$ such that $j-i\geq 2$, and that $r_{k,l}\geq N-e_l+e_k$ for all $k<l$ such that $l-k<j-i$. In particular, we can choose an index $k$ such that $i<k<j$, and the previous estimate holds for $r_{i,k}$ and $r_{k,j}$. But then, the inequality (\ref{recin}), applied to the quadruple $i<k=k<j$ yields
$$2N-1-e_j+e_i=r_{i,j}+r_{k,k}\geq r_{i,k}+r_{k,j}=2N-e_j+e_i,$$
a contradiction. We thus find an index $i$ such that $r_{i,i+1}=N-1-e_{i+1}+e_i$, and thus $r_{k,l}\leq N-1-e_{i+1}+e_i$ for all $k\leq i<i+1\leq l$ trivially. On the other hand, it is easy to compute the rank collection of $M({\mathbf a}^i)$ as
$$r_{j,k}(M({\mathbf a}^i))=N-1-e_{i+1}+e_i\mbox{ for }j\leq i<k,$$
and $r_{j,k}(M({\mathbf a}^i))=N$ otherwise. This proves that ${\mathbf r}\leq{\mathbf r}(M({\mathbf a}^i))$ as claimed.
\par
Now suppose that $M$ does not degenerate to $M^2$, and again consider the rank collection ${\mathbf r}={\mathbf r}(M)\not\geq {\mathbf r}^2$. We thus find a pair $i<j$  such that $$r_{i,j}\leq N-2-e_j+e_i.$$
We assume this pair to be chosen such that $j-i$ is minimal with this property; thus
$$r_{k,l}\geq N-1-e_l+e_k\mbox{ for all }k<l\mbox{ such that }l-k<j-i.$$ For every $i<k<j$, application of the inequality (\ref{recin}) to the quadruple $i<k=k<j$ yields
$$2N-2-e_j+e_i=N-2-e_j+e_i+N\geq r_{i,j}+r_{k,k}\geq$$
$$\geq r_{i,k}+r_{k,j}\geq N-1-e_k+e_i+N-1-e_j+e_j=2N-2-e_j+e_i,$$
from which we conclude
$$r_{i,k}=N-1-e_k+e_i,\, r_{k,j}=N-1-e_j+e_k\mbox{ for all }i<k<j$$
and
$$r_{i,j}=N-2-e_j+e_i.$$
Now we claim that
$$r_{k,l}=N-e_l+e_k\mbox{ for all }i<k<l<j.$$
This condition is empty if $j-i=1$, thus we can assume $j-i\geq 2$.
We prove this by induction over $k$, starting with $k=i+1$. For every $i+1<l<j$, application of (\ref{recin}) to $i<l-1<l<l$ yields
$$r_{i+1,l-1}=r_{i+1,l-1}+r_{i,l}-r_{i,l-1}+e_l-e_{l-1}\geq r_{i+1,l}+e_l-e_{l-1}.$$
This, together with (\ref{recin}) for $i<i+1\leq j-1<j$, yields the estimate
$$N=r_{i+1,i+1}\geq r_{i+1,i+2}+e_{i+2}-e_{i+1}\geq r_{i+1,i+3}+e_{i+3}-e_{i+1}\geq\ldots$$
$$\ldots\geq r_{i+1,j-1}+e_{j-1}-e_{i+1}\geq r_{i+1,j}+r_{i,j-1}-r_{i,j}+e_{j-1}-e_{i+1}=N,$$
thus equality everywhere. Now assume that $k>i+1$, and that the claim holds for all relevant $r_{k-1,l}$. Similarly to the previous argument, we arrive at an estimate
$$N=r_{k,k}\geq r_{k,k+1}+e_{k+1}-e_k\geq r_{k,k+2}+e_{k+2}-e_k\geq\ldots$$
$$\ldots\geq r_{k,j-1}+e_{j-1}-e_k\geq r_{k,j}+r_{k-1,j-1}-r_{k-1,j}+e_{j-1}-e_k=N,$$
and this again yields equality everywhere. This proves the claim.

Finally, we have the trivial estimates
\begin{itemize}
\item $r_{k,l}\leq r_{i,j}=N-2-e_j+e_i$ if $k\leq i\leq j\leq l$,
\item $r_{k,l}\leq r_{i,l}=N-1-e_l+e_i$ if $k<i<l<j$,
\item $r_{k,l}\leq r_{k,j}=N-1-e_j+e_k$ if $i<k<j<l$, and trivially
\item $r_{k,l}\leq N$ otherwise, that is, if $k<l\leq i<j$ or $i<j\leq k<l$.
\end{itemize}
An elementary calculation of ${\mathbf r}(M({\mathbf a}^{i,j}))$ shows that all these estimates together prove that
$${\mathbf r}\leq{\mathbf r}(M({\mathbf a}^{i,j})).$$ The theorem is proved.
\end{proof}

\subsection{Proof of Theorem~A}
We can now combine Theorem \ref{tc} and Theorem \ref{t5} to prove Theorem~A stated in the introduction. For the reader's convenience we restate it here. Let $Q$ be the equi-oriented quiver of type $A_n$. Let ${\mathbf d}=(N,\cdots, N)$, $\mathbf{e}=(e_1\leq\cdots\leq e_n)$ and $\mathbf{f}=\mathbf{d-e}$ be dimension vectors as above. Let $\pi: Y_\mathbf{e}\rightarrow R_\mathbf{d}$ be the universal quiver Grassmannian, whose generic fiber is a partial flag variety of dimension $\langle\mathbf{e,d-e}\rangle$. Consider the rank collections $\mathbf{r}^0$, $\mathbf{r}^1$ and $\mathbf{r}^2$ defined by
\begin{eqnarray*}
r^0_{i,j}&=&N,\ 1\le i<j\le n;\\
r^1_{i,j}&=&N-e_j+e_i,\ 1\le i<j\le n;\\
r^2_{i,j}&=&N-1-e_j+e_i,\ 1\le i<j\le n.
\end{eqnarray*}

\begin{theorem}\label{t2} The following holds:
\begin{itemize}\item[a)] The flat locus $U_{flat}\subset R_\mathbf{d}$ is the union of all orbits $\mathcal{O}_{{\mathbf r}}$  degenerating to 
$\mathcal{O}_{{\mathbf r^2}}$, i.e. $r_{i,j}\ge r^2_{i,j}$ for all pairs $i,j$.
\item[b)] The flat irreducible locus $U_{flat,irr}\subset R_\mathbf{d}$ is the union of all orbits $\mathcal{O}_{{\mathbf r}}$ 
degenerating to $\mathcal{O}_{{\mathbf r^1}}$, i.e. $r_{i,j}\ge r^1_{i,j}$ for all pairs $i,j$.
\end{itemize}
\end{theorem}
\begin{proof}
The flat locus $U_{flat}\subset R_\mathbf{d}$ consists of those $M\in R_\mathbf{d}$ such that the fiber $\pi^{-1}(M)$ has minimal dimension given by $\textrm{dim}\,\Gr_\mathbf{e}(M)=\langle\mathbf{e,d-e}\rangle$ (see e.g. \cite[Theorem~2~(1)]{CFFFR}). Let us prove that $\textrm{dim}\,\Gr_\mathbf{e}(M^2)=\langle\mathbf{e,d-e}\rangle$.
We have $M^2=P\oplus X$ with $P=P^{\mathbf e}$ and $X=S\oplus I^{\mathbf f}/S$ and we can apply the criterion of Theorem \ref{tc}. Using the exact sequence
$$0\rightarrow S\rightarrow I^{\mathbf f}\rightarrow I^{\mathbf f}/S\rightarrow 0,$$
and injectivity of $I^{\mathbf f}$, we can rewrite
$$\dim{\rm Hom}(\overline{K},S\oplus I^{\mathbf f}/S)-\dim{\rm Hom}(\overline{K},I^{\mathbf f})=\dim{\rm Ext}^1(\overline{K},S).$$
We thus have to check the inequality
$$\dim{\rm End}(\overline{K})\geq\dim{\rm Ext}^1(\overline{K},S).$$
Writing
$$\overline{K}=\bigoplus_{1\leq i\leq j<n}U_{i,j}^{k_{i,j}},$$
we have
$$\dim{\rm Ext}^1(\overline{K},S)=\sum_{1\leq i\leq j<n}k_{i,j},$$
and certainly
$$\dim{\rm End}(\overline{K})\geq \sum_{1\leq i\leq j<n}k_{i,j}^2.$$
This proves the claim about the dimension of ${\rm Gr}_{\mathbf e}(M^2)$.  Next, suppose that $M$ does not degenerate to $M^2$. 
By Theorem \ref{t5}, $M$ is a degeneration of some $M({\mathbf a}^{i,j})$ for $i<j$. We claim that 
${\rm Gr}_{\mathbf e}(M({\mathbf a}^{i,j}))$ has dimension strictly bigger than $\langle\mathbf{e},\mathbf{d-e}\rangle$. 
Namely, we can choose a subrepresentation $K\in{\rm Gr}_{\mathbf e}(M({\mathbf a}^{i,j}))$ such that 
$\overline{K}=U_{i,j-1}$ (notation as in Theorem~\ref{tc}). The conditions of Theorem \ref{tc} are easily seen to be violated. By upper
semi-continuity of fiber dimensions, $\dim{\rm Gr}_{\mathbf e}(M)$ is also strictly bigger than $\langle\mathbf{e,d-e}\rangle$.

Since $r^1_{i,j}\geq r^2_{i,j}$ for every $i,j$, it follows that $M^1\leq_{deg} M^2$ and hence $M^1\in U_{flat}$. 
Let us prove that $\Gr_\mathbf{e}(M^1)$ is irreducibile. This follows from Theorem~\ref{tc}: Indeed, $M^1=P\oplus X$ for $P=A$ and $X=A^*$. The criterion of Theorem \ref{tc} then reads $\dim{\rm End}(\overline{K})\geq 0$ which is trivially fulfilled, and irreducibility 
follows since $\overline{K}=0$ is the only representations for which equality holds. On the other hand, since 
${\rm Gr}_{\mathbf e}(M^1)$ is irreducible, then ${\rm Gr}_{\mathbf e}(M')$ is irreducible for every representation degenerating to $M^1$ 
(see e.g. \cite[Theorem~2~(2)]{CFFFR}). Suppose that $M$ does not degenerate to $M^1$. By Theorem \ref{t5}, $M$ is a degeneration of some 
$M({\mathbf a}^{i})$. We claim that ${\rm Gr}_{\mathbf e}(M({\mathbf a}^{i}))$ is reducible. Namely, we consider the two subrepresentations 
$K_1$ and $K_2$ determined by $\overline{K_1}=0$ and $\overline{K_2}=S_i$ (notation as in Theorem~\ref{tc}). Both $K_1$ and $K_2$ 
fulfill equality in the estimate of Theorem \ref{tc}, thus ${\rm Gr}_{\mathbf e}(M({\mathbf a}^i))$ has at least two irreducible components. It hence follows that ${\rm Gr}_{\mathbf e}(M)$ is reducible (see e.g. \cite[Theorem~2~(2)]{CFFFR}). 
\end{proof}

Since the orbit $\mathcal{O}_{{\mathbf r}^2}$ is minimal in the flat locus $U_{\rm flat}$, the linear degenerate partial 
flag variety ${\rm Fl}^{{\mathbf r}^2}(V)$ is maximally degenerated, thus we call it the maximally flat (mf)--linear 
degeneration of the partial flag variety. That this variety is rather natural, although being highly reducible and singular, 
is suggested by the next result (see also \cite{R,Rior}):

\begin{theorem}\label{t1}
The variety ${\rm Fl}^{{\mathbf r}^2}(V)$ is equi-dimensional, its number of irreducible components being the $n$-th Catalan number.
\end{theorem}

An arc diagram on $n$ points is a subset $A$ of $\{(i,j),\, 1\leq i<j\leq n\}$ (draw an arc from $i$ to $j$ for every element $(i,j)$ of $A$). An arc diagram $A$ is called non-crossing if there is no pair of different elements $(i,j)$, $(k,l)$ in $A$ such that 
$i\le k<j\le l$ 
(that is, two arcs are not allowed to properly cross, or to have the same left or right point. But immediate succession of arcs, like for example $\{(1,2),(2,3)\}$, is allowed).

To a non-crossing arc diagram we associate a rank collection ${\mathbf r}(A)$ by
$$r(A)_{i,j}=e_i-\#\{\mbox{arcs in $A$ starting in $[1,i]$ and ending in $[i+1,j]$}\}.$$
Define $S_A\subset {\rm Fl}^{{\mathbf r}^2}(V)$ as the set of all tuples $(U_1,\ldots,U_n)$ such that
$${\rm rank}((f_{j-1}\circ\ldots\circ f_i)|_{U_i}:U_i\rightarrow U_j)=r(A)_{i,j}$$
for all $i<j$.

Moreover, define representations $\overline{N}_A$ and $N_A$ of $Q$ by
$$\overline{N}_A=\bigoplus_{(i,j)\in A}U_{i,j-1},\; N_A=\bigoplus_i P_i^{c_i}\oplus \overline{N}_A,$$
where
$$c_i=e_i-e_{i-1}+\#\{\mbox{arcs ending in $i$}\}-\#\{\mbox{arcs starting in $i$}\}.$$
It is immediately verified that ${\mathbf r}(A)$ is precisely the rank collection of $N_A$.

We have the following more precise version of the previous theorem:

\begin{theorem}\label{t3} The irreducible components of ${\rm Fl}^{{\mathbf r}^2}(V)$ are the closures of the $S_A$, for $A$ a non-crossing arc diagram.
\end{theorem}

\begin{proof}
Working again in the setup and the notation of the proof of Theorem \ref{t2}, the irreducible components are parametrized by the representations $K$ as above for which the direct summand $\overline{K}$ satisfies
$$\dim{\rm End}(\overline{K})=\dim{\rm Ext}^1(\overline{K},S).$$
To satisfy this equality, it is thus necessary and sufficient for $\overline{K}$ to have all multiplicities $k_{i,j}$ of 
indecomposables equal to either $0$ or $1$, and there should be no non-zero maps between those $U_{i,j}$ for which $k_{i,j}=1$. 
But this can be made explicit since
$$\dim{\rm Hom}(U_{i,j},U_{m,l})=1\mbox{ if }m\leq i\leq l\leq j,$$
and zero otherwise. Thus $\overline{K}$ has to be of the form
$$\overline{K}=\bigoplus_{(i,j)\in I}U_{i,j-1}$$
for a set $I$ of pairs $(i,j)$ with $i\leq j$, such that there is no pair of different elements $(i,j),(m,l)\in I$ 
fulfilling $i\leq m<j\leq l$. These are precisely the representations $\overline{K}_A$ associated to non-crossing arc diagrams 
introduced above. It suffices to check that these $\overline{K}$ fulfill the additional assumptions, that is, that they embed into 
$S\oplus I^{\mathbf f}/S$ and the condition on dimension vectors. But this is easily verified.
\end{proof}

\section{Counting orbits in the flat irreducible locus: proof of Theorem~B}\label{Sec:ProofThmB}
We retain notation as in the previous sections. Thus, $Q$ is the equi-oriented quiver of type $A_n$, $N\geq n+1$, $\mathbf{d}=(N,\cdots, N)\in\mathbb{N}^n$, $\mathbf{e}=(e_1\leq\cdots\leq e_n)\in\mathbb{N}^n$, $e_0=0$, $e_n<N$ and $\overline{e}_{i+1}=e_{i+1}-e_i$.
Let $B_\mathbf{e}$
be the number of orbits in the flat irreducible locus in $R_\mathbf{d}$ (relative to the universal quiver Grassmannian $\pi: Y_\mathbf{e}\rightarrow R_\mathbf{d}$). 
\begin{lemma}
$B_{\mathbf e}$ does not depend on $N$, provided $N>e_n$.  
\end{lemma}
\begin{proof}
An orbit $\mathcal{O}_{{\mathbf r}}$ sits in the flat irreducible locus if and only if $r_{i,j}\ge r_{i,j}^1$ for all
pairs $i,j$ where $r_{i,j}^1=N-e_j+e_i$. Since $r_{i,j}^1$ can not exceed $N$, the number of orbits depends on
${\mathbf e}$, but not on $N$. 
\end{proof}

We consider the generating series 
$$B_n(x_1,\dots,x_n):=\sum_{\mathbf{e}=(e_1\leq\cdots\leq e_n)} B_{\mathbf e}\, x_1^{\overline{e}_1}x_2^{\overline{e}_2}x_3^{\overline{e}_3}\dots x_n^{\overline{e}_n}.$$
\begin{theorem}
We have
\[
B_n(x_1,\dots,x_n)=\prod_{i=1}^n (1-x_i)^{-1}\prod_{\emptyset\neq I\subseteq\{2,\dots,n\}} (1-x_I)^{-1},
\]
where for $I=\{i_1,\cdots,i_k\}$, $x_I:=x_{i_1}\cdots x_{i_k}$.
\end{theorem}

\begin{proof}
The proof is executed by induction on $n$. The  case $n=1$ is trivial. For $n=2$ one has $B_{\mathbf e}=e_2-e_1+1$ and
\[
B_n(x_1,x_2)=\sum_{0\le e_1 \le e_2} (e_2-e_1+1)x_1^{e_1} x_2^{e_2-e_1}=(1-x_1)^{-1}(1-x_2)^{-2}.
\]
 By induction, it suffices to show that
$$B_n(x_1,\dots,x_n)=B_{n-1}(x_1,\dots,x_{n-1})\times(1-x_n)^{-1}\prod_{I\subseteq\{2,\cdots,n-1\}}(1-x_Ix_n)^{-1},$$
where $x_\emptyset:=1$.

We fix the following notation:
\begin{itemize}
\item $\mathcal{R}$ is the set of rank collections $\mathbf{r}$ satisfying $\mathbf{r}^0\geq\mathbf{r}\geq\mathbf{r}^1$;
\item $\mathcal{P}_{n-1}$ is the power set on $\{1,2,\cdots,n-1\}$, $\mathcal{P}_{n-1}^*:=\mathcal{P}_{n-1}\setminus\{\emptyset\}$ and for $1\leq i\leq n-1$, $\mathcal{P}_{n-1}^i:=\{I\in\mathcal{P}_{n-1}\mid i\in I\}$;
\item $Q_\mathbf{e}$ is the polytope
$$Q_{\mathbf{e}}:=\left\{(f_I)\in\mathbb{R}_{\geq 0}^{\mathcal{P}_{n-1}^*}\mid\sum_{I\in\mathcal{P}_{n-1}^i}f_I\leq e_{i+1}-e_i,\ i=1,\cdots,n-1\right\};$$
\item for a polytope $P\subseteq\mathbb{R}^k$, we denote $P^{\mathbb{Z}}:=P\cap\mathbb{Z}^{k}\subset P$ the set of lattice points.
\end{itemize}

First notice that by Theorem \ref{t2}, $B_\mathbf{e}=\#\mathcal{R}$.  By definition, $Q_{\mathbf{e}}$ depends only on the mutual differences $\overline{e}_{i+1}$; we sometimes denote by $\overline{\mathbf{e}}=(\overline{e}_1,\overline{e}_2,\cdots,\overline{e}_n)$ the dimension vector of those differences. A rank collection $\mathbf{r}=(r_{i,j})$ satisfies this condition if and only if for $i=1,\cdots,n-1$, $r_{i,i+1}\geq N-\overline{e}_{i+1}$: the conditions posed on $r_{i,j}$ are implications of those on $r_{i,i+1}$.

We claim that there exists a bijection between $\mathcal{R}$ and $Q_\mathbf{e}^\mathbb{Z}$. To show this it suffices to establish two mutually inverse maps. 

\begin{itemize}
\item Given $\mathbf{f}:=(f_I)\in Q_\mathbf{e}^\mathbb{Z}$, we define for $1\leq i<j\leq n$ 
$$r_{i,j}(\mathbf{f}):=N-\sum_{I\in\mathcal{P}_{n-1}^*,\, I\cap[i,j]\neq\emptyset}f_I.$$
The defining inequalities of $\mathbf{f}$ imply that for any $i=1,2,\cdots,n-1$ one has $r_{i,i+1}(\mathbf{f})\geq N-\overline{e}_{i+1}$. This gives a rank collection in $\mathcal{R}$.
\item Conversely, let $\mathbf{r}\in\mathcal{R}$ be a rank collection. Let $(\mathrm{pr}_{J_1},\cdots,\mathrm{pr}_{J_{n-1}})\in R_\mathbf{d}$ be a projection sequence having rank collection $\mathbf{r}$. By assumption, $J_k\subset\{1,2,\cdots,N\}$ and $\#J_k\leq \overline{e}_{k+1}$. We associate to this projection sequence a point $\mathbf{f}=(f_I)\in Q_{\mathbf{e}}$ in the following way: for $k=1,\cdots,N$, we denote 
$$L_k:=\{s\mid 1\leq s\leq n-1,\ k\in J_s\}$$
and 
$$f_I:=\#\{k\mid 1\leq k\leq N,\ I=L_k\}.$$
It is clear that $f_I$ does not depend on the choice of the projection sequence. To show they give mutually inverse maps, it suffices to notice that for a projection sequence $(\mathrm{pr}_{J_1},\cdots,\mathrm{pr}_{J_{n-1}})$, the rank $r_{i,j}=N-\#J_i\cup\cdots\cup J_{j-1}$ and 
$$\#J_i\cup\cdots\cup J_{j-1}=\#\bigcup_{\substack{I\in\mathcal{P}_{n-1}^*\\ I\cap[i,j]\neq\emptyset}}\{k\mid 1\leq k\leq N,\ I=L_k\}.$$

\end{itemize}

For $\mathbf{f}\in Q_\mathbf{e}$ and $1\leq i\neq j\leq n-1$, we denote
$$\Sigma_i(\mathbf{f}):=\sum_{I\in\mathcal{P}_{n-1}^i}f_I,\ \ \mathcal{P}_{n-1}^{i,j}=\mathcal{P}_{n-1}^{i}\cap \mathcal{P}_{n-1}^{j},\ \text{and} \ \Sigma_{i,j}(\mathbf{f}):=\sum_{I\in\mathcal{P}_{n-1}^{i,j}}f_I.$$

Then 
$$B_n(x_1,\cdots,x_n)=\sum_{\mathbf{e}=(e_1\leq\cdots\leq e_n)\in\mathbb{N}^n}\#Q_\mathbf{e}^\mathbb{Z}\,x^{\overline{\mathbf{e}}}=\sum_{\substack{\Sigma_i(\mathbf{f})\leq \overline{e}_{i+1}\\ i=1,\cdots,n-1}}x^{\overline{\mathbf{e}}}.
$$

We consider the projected and the fibre polytopes. Let $\pi:\mathbb{R}^{\mathcal{P}_{n-1}^*}\ra\mathbb{R}^{\mathcal{P}_{n-1}^{n-1}}$ be the linear projection induced by the inclusion $\mathcal{P}_{n-1}^{n-1}\subseteq\mathcal{P}_{n-1}^*$. We denote $Q_{\mathbf{e},n-1}:=\pi(Q_\mathbf{e})$, and for $\mathbf{g}\in Q_{\mathbf{e},n-1}$, the fibre polytope is denoted by $Q_\mathbf{e}(\mathbf{g}):=\pi^{-1}(\mathbf{g})\cap Q_\mathbf{e}$.

By rearranging the sum we have
$$\sum_{\overline{e}_n\geq 0}\sum_{\mathbf{g}\in Q_{\mathbf{e},n-1}}\left(\sum_{\overline{e}_1,\cdots,\overline{e}_{n-1}\geq 0} \sum_{\mathbf{h}\in Q_\mathbf{e}(\mathbf{g})} x_1^{\overline{e}_1}x_2^{\overline{e}_2-\Sigma_{1,n-1}(\mathbf{g})}\cdots x_{n-1}^{\overline{e}_{n-1}-\Sigma_{n-2,n-1}(\mathbf{g})}\right)\times
$$
$$\times x_2^{\Sigma_{1,n-1}(\mathbf{g})}\cdots x_{n-1}^{\Sigma_{n-2,n-1}(\mathbf{g})}x_n^{\overline{e}_n}.$$

The bracket in the middle gives $B_{n-1}(x_1,\cdots,x_{n-1})$. It suffices to evaluate the sum
$$\sum_{\overline{e}_n\geq 0}\sum_{\mathbf{g}\in Q_{\mathbf{e},n-1}}x_2^{\Sigma_{1,n-1}(\mathbf{g})}\cdots x_{n-1}^{\Sigma_{n-2,n-1}(\mathbf{g})}x_n^{\overline{e}_n},$$
which can be written into 
$$\sum_{\substack{g_I\geq 0\\ I=\{i_1<\cdots<i_k\}\in\mathcal{P}_{n-1}^{n-1}}}\left(\sum_{\overline{e}_n-\Sigma_{n-1}(\mathbf{g})\geq 0}x_n^{\overline{e}_n-\Sigma_{n-1}(\mathbf{g})}\right)(x_{i_1+1}\cdots x_{i_k+1})^{g_I}.$$
Notice that in the first sum, $i_k=n-1$ hence the last variable $x_{i_k+1}=x_n$. The sum in the middle bracket gives $(1-x_n)^{-1}$; for the remaining summation, it suffices to notice that the variables $g_I$ are independent, hence we obtain
$$(1-x_n)^{-1}\prod_{I\subseteq\{2,\cdots,n-1\}}(1-x_Ix_n)^{-1},$$
and the proof terminates.
\end{proof}

From \cite[Section 4.2]{CFFFR}, the Bell numbers can be recovered as
$$\left.\frac{\partial^n}{\partial x_1\cdots\partial x_n}\right|_{x_1=\cdots=x_n=0}B_n(x_1,\cdots,x_n).$$
In fact, the coefficient in front of $x_1\dots x_n$ in $B_n$ is equal to the number of orbits in the flat
irreducible locus corresponding to the case of complete flags ($e_1=1,\dots,e_n=n$).

\section{Homogeneous coordinate rings: flat locus}\label{Sec:Plucker}
We start with linear degenerations of the \emph{complete} flag variety. 
Thus, $Q$ denotes the equi-oriented quiver of type $A_n$, $N=n+1$, $\mathbf{d}=(n+1, n+1, \cdots, n+1)\in\mathbb{N}^n$ and $\mathbf{e}=(1,2,\cdots, n)$. Moreover, $\pi: Y_\mathbf{e}\rightarrow R_\mathbf{d}$ is the universal quiver Grassmannian whose generic fiber is the complete flag variety  of dimension $\frac{n(n+1)}{2}$, and all other fibers are quiver Grassmannians $\Gr_\mathbf{e}(M)$ where $M\in R_\mathbf{d}$.
We consider the Pl\"ucker embedding 
${\rm Gr}_{\mathbf e}(M)\subset \prod_{i=1}^n {\mathbb P} (\Lambda^{i} M_i)$. 
Our goal is to describe the reduced scheme structure of the embedded Grassmannian in the flat irreducible locus, i.e. to describe 
the ideal of multi-homogeneous polynomials vanishing on the image of Grassmannians in an orbit degenerating
to $\mathcal{O}_{{\mathbf r}^1}$. The strategy is as follows:
first, we give explicit set of Pl\"ucker-like quadratic relations.
Second, we show that for any orbit $\mathcal{O}$ degenerating to $\mathcal{O}_{{\mathbf r}^1}$ there exists a point $M\in \mathcal{O}$ such
that these relations are enough to express any monomial (in Pl\"ucker coordinates) from the coordinate ring of 
${\rm Gr}_{\mathbf e}(M)$ in terms of PBW semi-standard monomials. This would imply that our quadratic relations indeed provide the
reduced scheme structure due to the fact that the number of PBW semi-standard monomials of shape $\lambda$ is equal to
the dimension of the irreducible $\mathrm{SL}_N$ module of highest weight $\lambda$ (recall that degenerations
over $\mathcal{O}_{{\mathbf r}^1}$ -- even over $\mathcal{O}_{{\mathbf r}^2}$ -- are flat).        

\begin{remark}
The results in the following two subsections hold for the whole flat locus. In particular, the set-theoretic equality 
(Proposition \ref{set-theoretic}) of the quiver Grassmannian and the vanishing set of the Pl\"ucker-like quadratic relations 
are true for the whole flat locus. In Section \ref{bases-ring}, the crucial ingredient is the existence of a special point in 
every orbit (Lemma \ref{canon}), which can be shown to exist for orbits in the flat, irreducible locus and a few other orbits 
(see Remark \ref{Remarkr2}). Nevertheless, we conjecture that Theorem \ref{ssPBW} extend to the whole flat locus. 
\end{remark}

\subsection{Degenerate Pl\"ucker relations for the complete flags}
We first fix some notation:
\begin{enumerate}
\item for $n\in\mathbb{N}_{>0}$, $[n]=\{1,2,\ldots,n\}$;
\item $I(d,n)$ is the set of all $d$-element subsets of $[n]$; $T(d,n)$ is the set of $d$-tuples $(j_1,\ldots,j_d)$ with $1\leq j_1,\ldots,j_d\leq n$ pairwise distinct.
\end{enumerate}

We fix a basis $\{v_1,v_2,\ldots,v_{n+1}\}$ to identify $V$ with $\mc^{n+1}$. 
Let $I_1,\ldots,I_{n-1}$ be subsets of $[n+1]$ and $\pr_{I_k}:\mc^{n+1}\ra\mc^{n+1}$ be the projection along basis elements indexed by $I_k$.
Let $M$ be the following representation of $Q$:
$$
M:\xymatrix{
\mc^{n+1}\ar^-{\pr_{I_1}}[r]&\mc^{n+1}\ar^-{\pr_{I_2}}[r]&\ar[r]\cdots\ar^-{\pr_{I_{n-1}}}[r]&\mc^{n+1}.
}
$$
Assume that $I_1,\ldots,I_{n-1}$ are chosen such that the dimension of the quiver Grassmannian $\dim\Gr_{\mathbf{e}}(M)=\frac{n(n+1)}{2}$ is minimal (i.e. $M\in U_{flat}$).

We fix the Pl\"ucker embedding of the quiver Grassmannian:
$$\Gr_{\mathbf{e}}(M)\hookrightarrow \Gr_1(\mc^{n+1})\times\Gr_2(\mc^{n+1})\times\cdots\times\Gr_n(\mc^{n+1})
\hookrightarrow\prod_{k=1}^n\mathbb{P}(\Lambda^k\mc^{n+1}).$$
For $I\in I(d,n+1)$, let $X_I$ be the Pl\"ucker coordinate on $\Gr_d(\mc^{n+1})$. Let $\mathcal{A}=\mc[X_I\mid I\in I(d,n+1)
\text{ for some }1\leq d\leq n]$ and $\mathcal{A}_t:=\mathcal{A}[t]$.

We first introduce the deformed Pl\"ucker relations with respect to a set $\emptyset\neq K\subset [n+1]$. For $J\in I(r,n+1)$, we define $\deg_K(X_J):=\#(K\cap J)$.

For $J=(j_1,j_2,\ldots,j_r)\in T(r,n+1)$, $L=(l_1,l_2,\ldots,l_s)\in T(s,n+1)$ with $1\leq s<r\leq n$ and $1\leq k\leq s$, we denote
$$R_{J,L;k}^K(t):=t^{-m(J,L,K)}\left(t^{\deg_K(X_L)}X_JX_L-\sum_{1\leq\alpha_1<\cdots<\alpha_k\leq r}t^{\deg_K(X_{L_\alpha})}X_{J_\alpha}X_{L_\alpha}\right)\in\mathcal{A}_t,$$
where for $\alpha=(\alpha_1,\cdots,\alpha_k)$ with $1\leq \alpha_1<\cdots<\alpha_k\leq r$, 
$$J_\alpha=(j_1,\cdots,j_{\alpha_1-1},l_1,j_{\alpha_1+1},\cdots,j_{\alpha_2-1},l_2,j_{\alpha_2+1},\cdots,j_r),$$ 
$$L_\alpha=(j_{\alpha_1},j_{\alpha_2},\cdots, j_{\alpha_k},l_{k+1},\cdots,l_s);$$ 
and 
$$m(J,L,K)=\min\{\deg_K(X_L),\deg_K(X_{L_\alpha})\mid 1\leq\alpha_1<\cdots<\alpha_k\leq r\}.$$

In particular, when the projection sequence $\mathbf{I}=(I_1,I_2,\ldots,I_{n-1})$ is given, we define for $1\leq s<r\leq n$ a set $K(s,r)\subset [n+1]$ by:
$$K(s,r):=I_s\cup I_{s+1}\cup\cdots\cup I_{r-1}.$$
Then $\pr_{K(s,r)}=\pr_{I_{r-1}}\circ\cdots\circ\pr_{I_s}:\mc^{n+1}\ra\mc^{n+1}$.

\begin{definition}\label{ideal}
Let $\mathfrak{I}_\mathbf{I}$ be the ideal in $\mathcal{A}$ generated by the following relations: 
\begin{enumerate}
\item[(P1)] Pl\"ucker relations in $\Gr_k(\mc^{n+1})$ for $1\leq k\leq n$;
\item[(P2)] for any $1\leq s<r\leq n$, $J\in T(r,n+1)$, $L\in T(s,n+1)$ and $1\leq k\leq\max\{1,\#(L\backslash K(s,r))\}$, 
the relation $R_{J,L;k}^{K(s,r)}(0)$.
\end{enumerate}
\end{definition}

Let $X_\mathbf{I}=V(\mathfrak{I}_\mathbf{I})$ denote the vanishing locus of $\mathfrak{I}_\mathbf{I}$ 
in $\prod_{k=1}^n\mathbb{P}(\Lambda^k\mc^{n+1})$.

\begin{remark}\label{Rmk:Pluecker}
In the study of these relations, we can always assume that $L$ is not contained in $K$ for $K=K(s,r)$. Under the assumption $\dim\Gr_\mathbf{e}(M)=\frac{n(n+1)}{2}$, we have $\#K\leq r-s+1$. If $J\subseteq K$, $s$ must be $1$ and hence $J=K$. In this case $L\subset K$ will make the relation $R_{J,L;1}^K(0)$ to be empty. 
\par
Without loss of generality we can assume that $j_1\notin K$. If $L\subseteq K$ we denote $\widetilde{J}:=(l_1,j_2,\cdots,j_r)$ and $\widetilde{L}:=(j_1,l_2,\cdots,l_s)$, then $\widetilde{L}$ is not contained in $K$ and $R_{J,L;1}^K(0)=-R_{\widetilde{J},\widetilde{L};1}^K(0)$.
\end{remark}

\begin{proposition}\label{set-theoretic}
The set $\Gr_{\mathbf{e}}(M)$ coincides with $X_\mathbf{I}$.
\end{proposition}

\begin{proof}
We first show that $\Gr_{\mathbf{e}}(M)\subset X_\mathbf{I}$. Let $\mathbf{x}=(V_1,V_2,\ldots,V_n)\in \Gr_{\mathbf{e}}(M)$ and 
$1\leq s<r\leq n$. For $J=(j_1,j_2,\ldots,j_r)\in T(r,n+1)$, $L=(l_1,l_2,\ldots,l_s)\in T(s,n+1)$, $K=K(s,r)$ and 
$1\leq k\leq\#(L\backslash K)$, one needs to show that $R_{J,L;k}^K(0)$ vanishes on $\mathbf{x}$. Since $k\leq\#(L\backslash K)$, 
by arranging elements in $L$ we can always assume that $l_1,\cdots,l_k\notin K$. With this assumption, the proof of Theorem 3.13 
in \cite{Feigin2} (or Proposition 2.2 in \cite{Feigin1}) can be applied.
\par
To show the other inclusion, we take $\mathbf{x}=(V_1,\cdots,V_n)\notin\Gr_{\mathbf{e}}(M)$ and construct a relation $R_{J,L;k}^K(0)(\mathbf{x})\neq 0$. According to the assumption, there exist $V_1\in\Gr_s(\mc^{n+1})$ and $V_2\in\Gr_r(\mc^{n+1})$ for $1\leq s<r\leq n$ such that $\pr_K(V_1)\nsubseteq V_2$. 
\par
We prove that $\mathbf{x}\notin X_\mathbf{I}$.
\par
Assume that $E_K=\sspan\{v_k\mid k\in K\}$ and $E_{K^c}=\sspan\{v_l\mid l\in [n+1]\backslash K\}$. We choose a basis $\{e_1,e_2,\ldots,e_s\}$ of $V_1$ in the following way: $e_{t+1},\ldots,e_s$ is a basis of $V_1\cap E_K$, then extend it to a basis $e_1,\ldots,e_s$ of $V_1$. Up to base changes in $E_K$ and $E_{K^c}$ we can assume that 
$$e_1=v_{l_1}+w_1,\ \ldots,\ e_t=v_{l_t}+w_t,\ \ e_{t+1}=v_{l_{t+1}},\ \ldots,\ e_s=v_{l_s},$$
where $w_1,\ldots,w_t\in V_1\cap E_K$. As $V_1\nsubseteq E_K$, we can assume that $l_1\notin K$.
\par
We denote $L=(l_1,\ldots,l_s)$, then $X_L(V_1)\neq 0$. There exists a tuple $J=(j_1,\ldots,j_r)$ such that $X_J(V_2)\neq 0$. 
\par
We consider the relation $t^{m(J,L,K)}R_{J,L;1}^K(t)$:
$$t^{m(J,L,K)}R_{J,L}^K(t):=t^{\#(L\cap K)}X_JX_L-\sum_{1\leq\alpha\leq r}t^{\#(L_\alpha\cap K)}X_{J_\alpha}X_{L_\alpha},$$
where $J_\alpha=(j_1,\cdots,j_{\alpha-1},l_1,j_{\alpha+1},\cdots,j_r)$ and $L_\alpha=(j_\alpha,l_2,\cdots,l_s)$. Since $l_1\notin K$, by definition, 
$$R_{J,L;1}^K(0)=X_JX_L-\sum_{1\leq\alpha\leq r,\,j_\alpha\notin K}X_{J_\alpha}X_{L_\alpha}.$$

We claim that for any $1\leq \alpha\leq r$ with $j_\alpha\notin K$, 
$$X_{j_\alpha,l_2,\ldots,l_s}(V_1)=0.$$
As $V_1\subset \sspan(\{v_{l_1},\ldots,v_{l_s}\}\cap\{v_k\mid k\in K\})$, it suffices to show that for $j_\alpha\in\{l_1,\ldots,l_t\}$ the above equality holds. But in this case the corresponding Pl\"ucker relation is empty.
\par
As a conclusion, $R_{J,L;1}^K(0)(\mathbf{x})=X_JX_L(\mathbf{x})\neq 0$.
\end{proof}

\begin{example}
Consider $M=\xymatrix{\mc^{4}\ar^-{\pr_{1,2}}[r]&\mc^{4}\ar^-{\pr_{2,3}}[r]&\mc^{4}}$: $\Gr_{\mathbf{e}}(M)$ is the mf-linear degenerate flag variety. The defining ideal is given by:
$$X_{12}X_3,\ \ X_{12}X_4,\ \ X_{13}X_4-X_{14}X_3,\ \ X_{23}X_4-X_{24}X_3,$$
$$X_{123}X_4,\ \ X_{123}X_{14},\ \ X_{123}X_{24}+X_{234}X_{12},\ \ X_{123}X_{34}+X_{234}X_{13},\ \ X_{234}X_{14},$$
$$X_{12}X_{34}-X_{13}X_{24}+X_{14}X_{23}.$$
\end{example}

\subsection{Straightening law}
We assume that $\mathbf{I}=(I_1,\cdots,I_{n-1})$ with $I_k\subset \{k,k+1\}$, then $K(s,r)\subset\{s,s+1,\cdots,r\}$. Recall that 
$$
\xymatrix{
M=\mc^{n+1}\ar^-{\pr_{I_1}}[r]&\mc^{n+1}\ar^-{\pr_{I_2}}[r]&\ar[r]\cdots\ar^-{\pr_{I_{n-1}}}[r]&\mc^{n+1}
}
$$
and $\dim\Gr_\mathbf{e}(M)=\frac{n(n+1)}{2}$.

A PBW semi-standard Young tableau \cite{Feigin2} of shape $\la=\sum_{i=1}^{N-1} m_i\om_i$ is a filling $T_{i,j}$ of the Young tableau
with $m_i$-columns of length $i$ ($i=1,\dots,N-1$) such that the following conditions are satisfied ($l_j$ denotes the length
of the $j$-th column):
\begin{itemize}
\item if $T_{i,j}\le l_j$, then $T_{i,j}=i$;
\item if $T_{i_1,j}, T_{i_2,j}>l_j$, then $i_1<i_2$ implies $T_{i_1,j}>T_{i_2,j}$;
\item for any $j>1$, $i\le l_j$ there exists $i_0\le l_{j-1}$ such that $T_{i_0,j-1}\ge T_{i,j}$. 
\end{itemize}
We call a monomial in Pl\"ucker coordinates PBW semi-standard if it corresponds to a PBW semi-standard Young tableau.

\begin{proposition}\label{P1P2}
The relations $(P1)$ and $(P2)$ are enough to express any monomial in Pl\"ucker coordinates 
on $\Gr_\mathbf{e}(M)$ as a linear combination of the  PBW semi-standard monomials.
\end{proposition}

\begin{proof}
We consider the following total ordering defined on the set of tableaux of a fixed shape: for two tableaux $T^{(1)}$ and $T^{(2)}$: we say $T^{(1)}\geq T^{(2)}$, if there exists $(i,j)$ such that for any $(k,\ell)$ where either $\ell>j$ or $\ell=j$ and $k>i$, $T_{k,\ell}^{(1)}=T_{k,\ell}^{(2)}$ and $T_{i,j}^{(1)}>T_{i,j}^{(2)}$.

Assume that we have a non-PBW semi-standard Young tableau with two columns $A$ and $B$ representing the product of Pl\"ucker monomial $X_AX_B$ where $r=\ell(A)\geq\ell(B)=s$ such that both $A$ and $B$ are PBW tableaux. 

We assume that $k_0$ is the smallest index such that for any $k\geq k_0$, $A_k<B_{k_0}$. First notice that by the semi-standard property, $B_{k_0}\geq r+1$. Assume that $A=(j_1,\cdots,j_r)$ and $B=(l_1,\cdots,l_s)$, then $j_k<l_{k_0}$. Since $l_{k_0}\geq r+1$, $l_1,\dots,l_{k_0-1}$ are either strictly less than $s$ or strictly larger than $r+1$; this implies that 
$$\{l_1,\cdots,l_{k_0}\}\cap \{s,s+1,\cdots,r\}=\emptyset$$
and hence $l_1,\cdots,l_{k_0}\notin K(s,r)$.

We consider the relation $R_{A,B;k_0}^{K(s,r)}(0)$ from $(P2)$ exchanging the first $k_0$ indices in $B$ with an 
arbitrary $k_0$ elements in $A$: the resulting tableaux are strictly smaller in the total order on tableaux introduced above. 
Moreover, the monomial $X_AX_B$ appears in the relation: assume that $X_{A'}X_{B'}$ is a monomial obtained from the exchange, 
then $\#(B'\cap K)\geq \#(B\cap K)$. As there are only finitely number of tableaux of a fixed shape, this procedure will 
terminate after having been repeated finitely many times.
\end{proof}

\subsection{Bases in the coordinate rings}\label{bases-ring}
\begin{lemma}\label{canon}
a) Let $N=n+1$ and $e_i=i$, $i=1,\dots,n$. Then for an orbit $\mathcal{O}$ degenerating to $\mathcal{O}_{{\mathbf r}^1}$ 
there exists a point $M\in\mathcal{O}$ such that 
the defining maps ${\mathbf f}=(f_1,\dots,f_{n-1})$, $f_i:M_i\to M_{i+1}$ satisfy the following properties:
\begin{itemize}
\item $(f_i)_{a,b}=1$ if $a=b<i$ or $a=b>i+1$,
\item $(f_i)_{a,b}=0$ if $a,b<i$, $a\ne b$;
\item $(f_i)_{a,b}=0$ if $a,b>i+1$, $a\ne b$;
\item $(f_i)_{a,b}=0$ if $a>b$.
\end{itemize} 
b) For a partial flag variety case (arbitrary $N$, $e_1,\dots,e_n$) any orbit has a canonical form, which is a projection
of the canonical form for the complete flags (with the same $N$) forgetting all the components but the ones
numbered by $e_1,\dots,e_n$. 
\end{lemma}
\begin{proof} This follows immediately from the definition of a transversal slice to the flat irreducible locus given in \cite[Proposition 3]{CFFFR}.
\end{proof}

\begin{theorem}\label{ssPBW}
For any orbit $\mathcal{O}$ degenerating to $\mathcal{O}_{{\mathbf r}^1}$ there exists a point $M\in \mathcal{O}$ such
that the semi-standard PBW tableaux provide a basis in the homogeneous coordinate ring of ${\rm Gr}_{\mathbf e}(M)$.
\end{theorem}
\begin{proof}
We consider a representation $M=(M_1,\dots,M_n)$ satisfying conditions of Lemma \ref{canon}. 
Let $f_{p,q}:M_p\to M_q$ be corresponding linear map. Let $\{v_1,\dots,v_N\}$ be the standard basis of $M_p$ and $M_q$
(the conditions from Lemma \ref{canon} are written for matrix elements of the maps $f_i$ in the basis $\{v_a\}$).
Since the orbit of $M$ degenerates to $\mathcal{O}_{{\mathbf r}^1}$, the corank of $f_{p,q}$ is at most $q-p$.
Let us choose a basis $\{v'_b\}$ of $M_p$ and $\{v''_b\}$ of $M_q$ such that the matrix of $f$ in these bases is   
$\pr_{I}$ for some $I$ with $|I|\le q-p$. Since the matrix $f_{p,q}$ in the basis $\{v_a\}$ is upper-triangular,
we may assume that the matrices expressing $\{v'_b\}$ and $\{v''_b\}$ in terms of the initial basis $\{v_a\}$ 
are both upper-triangular. 

Now assume we are given a non PBW semi-standard monomial $X_AX_B$, $|A|=p$, $|B|=q$ written in the coordinates corresponding to the
basis $\{v_a\}$. Let $Y_{A'}$, $Y_{B'}$ be the Pl\"ucker coordinates in the bases $\{v'_b\}$ and $\{v''_b\}$.
Then $X_AX_B-Y_AY_B$ is equal to the linear combination of monomials (in $X$-coordinates or in $Y$-coordinates)
such that the sum of all indices of these monomials is strictly smaller than that of $X_AX_B$.  
Since Proposition \ref{P1P2} tells us that a non PBW semi-standard $Y_AY_B$ can be rewritten in terms of the
PBW semi-standard quadratic monomials, the same is true for $X_AX_B$. 

Recall (see \cite{Feigin2}) that the PBW semi-standard monomials form a basis in the homogeneous coordinate ring of
the PBW degenerate flag variety, which is isomorphic to ${\rm Gr}_{\mathbf e}(K)$ for $K\in \mathcal{O}_{{\mathbf r}^1}$. 
Since the degeneration over the flat locus is flat, the dimension of the homogeneous components of the 
coordinate rings does not change in the family. We conclude that PBW semi-standard monomials form a basis in the homogeneous
coordinate ring of our quiver Grassmannian ${\rm Gr}_{\mathbf e}(M)$ and the relations from Definition \ref{ideal}
(after the base change as above) provide the reduced scheme structure.     
\end{proof}

\begin{remark}\label{Remarkr2}
Theorem \ref{ssPBW} holds for all partial flag varieties. The proof given above generalizes in a straightforward way
by forgetting the corresponding components. Moreover, the proof generalizes also to all orbits in the flat locus, that contain a point satisfying conditions in Lemma \ref{canon}. For example, in the ${\mathbf r}^2$-orbit, there is a point such that the semi-standard PBW tableaux provide a basis in the homogeneous coordinate ring.
\end{remark}

\section{Flat irreducible locus: group action and line bundles}\label{Sec:LineBundles}
\subsection{Lie algebras and representations}
Let $T\subset R$ be the transversal slice through the flat irreducible locus from \cite{CFFFR}, consisting
of all tuples of linear maps $(f_1,\ldots,f_{n-1})$ such that the matrix entry of $f_i$ in the standard basis $\{v_1,v_2,\cdots,v_{n+1}\}$ is given by:
$$(f_i)_{p,q}=\left\{\begin{array}{ccc}1&,& p=q\not=i+1,\\ \lambda_{p,q}&,& 2\leq p\leq i+1\leq q\leq n,\\
0&,& \mbox{ otherwise }\end{array}\right.$$
for certain $(\lambda_{i,j})_{2\leq i\leq j\leq n}$.
Let $M_t=((M_t)_1,\dots,(M_t)_n)$  be the representation of $Q$ corresponding to $t\in T$ and  
let $F_t$ denote the composition $f_{n-1}\circ f_{n-2}\circ\dots\circ f_1$.  
Then the matrix coefficient $(F_t)_{a,b}$ equals to $(b-a+1)\la_{a,b}$ if $2\leq a\le b\leq n$ and vanishes otherwise with the exception 
$(F_t)_{1,1}=(F_t)_{n+1,n+1}=1$.

Let $\fg_t$ be the Lie algebra of all $(n+1)\times (n+1)$ matrices with the bracket defined by  
the formula $[x,y]_t=xF_ty-yF_tx$. 
\begin{remark}
The subspace of upper triangular matrices $\fb_+$ is closed with respect to the bracket $[\cdot,\cdot]_t$.
However, this is {\it not} true for the subspace of strictly lower triangular matrices $\fn_-$.
\end{remark}

The deformed brackets naturally arise via endomorphism algebras
of $M_t$. Namely, let us define the family of maps $\Phi_t: \fg_t\to \mathrm {End} (M_t)$ by the formula
\[
(\Phi_t(x))_i=f_{i-1}\circ\dots \circ f_1\circ x\circ f_{n-1}\circ f_{n-2}\circ\dots \circ f_i. 
\]  
\begin{remark}
The condition that the $\Phi_t(x)$ indeed defines an endomorphism of the representation is easily verified, since 
this amounts to the conditions $f_i\circ (\Phi_t(x))_i= (\Phi_t(x))_{i+1}\circ f_i$ for $i < n$, which are immediate from the definition 
of the $\Phi_t$.
\end{remark}
Then we have the following lemma.
\begin{lemma}
The map $\Phi_t$ is a homomorphism of Lie algebras with respect to the bracket $[\cdot,\cdot]_t$ on $\fg_t$ and the usual
composition on $\mathrm{End} (M_t)$.    
\end{lemma} 

Thanks to the lemma above, the image of $\Phi_t$ is a Lie subalgebra in $\mathrm{End} (M_t)$. We denote this Lie subalgebra by 
$\fa_t$. 
\begin{lemma}\label{nokernel}
The map $\Phi_t$ has no kernel on $\fn_-$.
\end{lemma}
\begin{proof}
The lower left $(n-i)\times i$-submatrix of $\Phi_t(x)_i$ coincides with the lower left $(n-i)\times i$-submatrix of 
$x$, which means that we can recover $x$ completely from $\Phi_t(x)$.
\end{proof}

\begin{remark}
The dimension of $\Phi_t(\fb_+)$ does depend on $t$. For example, if $\la_{i,j}=\delta_{i,j}$, then 
$\dim\Phi_t(\fb_+)=\dim(\fb_+)=(n+1)(n+2)/2$. If all $\la_{i,j}=0$, then $\dim\Phi_t(\fb_+)=2n+1$.
\end{remark}

Let us construct a family of representations $V_t(\mu)$ of $\fa_t$ labeled by dominant integral weights
$\mu=m_1\om_1+\dots +m_n\om_n$ with $m_i\in\bZ_{\ge 0}$.
We start with the fundamental representations. 
\begin{definition}
For $k=1,\dots,n$ we define $V(\om_k)\subset \Lambda^k (M_t)_k$ as the ${\rm U}(\fa_t)$-span of the vector 
$v_{\omega_k}=v_1\wedge\dots\wedge v_k$. 
\end{definition}

\begin{lemma}
$V(\om_k)=\Lambda^k (M_t)_k$.
\end{lemma}
\begin{proof}
This is implied by the argument from the proof of Lemma \ref{nokernel}.
\end{proof}

\begin{definition}
For a dominant integral weight $\mu=\sum_{k=1}^n m_k\om_k$ we define the $\fa_t$-module 
$V_t(\mu)\subset \bigotimes V_t(\om_k)^{\otimes m_k}$ as the ${\rm U}(\fa_t)$-span of the vector
$v_\mu=\bigotimes v_{\om_k}^{\otimes m_k}$.  
\end{definition}

\begin{remark}\label{fn-}
Each space $V_t(\mu)$ is generated from the cyclic vector $v_\mu$ by the action of the (associative) algebra
of operators generated by $\Phi_t(\fn_-)$.  In fact, one easily sees that $\Phi_t(\fb_+) v_\mu\subset \bC v_\mu$.  
\end{remark}

In order to compute the dimension and to construct bases of the spaces $V_t(\mu)$ we define the
following total order on the standard basis $E_{a,b}$, $a>b$ of the algebra $\fn_-$ of strictly lower triangular
matrices: $E_{a,b}<E_{c,d}$ if $a-b>c-d$ or ($a-b=c-d$ and $a<c$). We extend this order to the homogeneous lexicographic order on the set of ordered monomials $E_{a_1,b_1}\dots E_{a_L,b_L}$, $E_{a_1,b_1}>\dots > E_{a_L,b_L}$. Namely, 
for two ordered monomials $E_{a_1,b_1}\dots E_{a_L,b_L}<E_{a'_1,b'_1}\dots E_{a'_M,b'_M}$ if $L<M$ or  
($L=M$ and there exists $j$ such that $E_{a_j,b_j}<E_{a'_j,b'_j}$ and $E_{a_i,b_i}=E_{a'_i,b'_i}$ for $i>j$).  
Given such an ordering we define monomial bases of $V_t(\mu)$ (see Remark \ref{fn-}) as follows.
We say that a vector $\prod_{i=1}^L E_{a_i,b_i}v_\mu\in V_t(\mu)$ is essential if 
\[
\prod_{i=1}^L E_{a_i,b_i}v_\mu\notin {\mathrm{span}} \left\{\prod_{i=1}^M E_{c_i,d_i}v_\mu\,\right| \left. \ 
\prod_{i=1}^M E_{c_i,d_i}<\prod_{i=1}^L E_{a_i,b_i}\right\}.
\] 
Clearly, the set of essential vectors form a basis of $V_t(\mu)$. 

For an element $\bs=(s_{i,j})_{1\le j<i\le n+1}$, $s_{i,j}\in \bZ_{\ge 0}$ we denote by $\bE^\bs$
the ordered product $\prod E_{i,j}^{s_{i,j}}$. Let $S_t(\mu)$ be the set of essential exponents, i.e. the set
of all $\bs$ such that ${\bE}^\bs v_\mu$ is an essential vector.

\begin{remark}
For $t=0$ (i.e. all $\la_{i,j}=0$)
the set of essential vectors is described via the combinatorics of Dyck paths (see \cite{FeFoLit}).
In particular, the number of essential vectors is equal to the dimension of the irreducible $\msl_{n+1}$-module 
$V(\mu)$ (which corresponds to $t$ with all $\la_{i,j}\ne 0$).    
\end{remark}

Our goal is to show that the set of essential monomials does not depend on $t$. In particular, we will show that 
$\dim V_t(\mu)$ is independent of $t$.

\begin{lemma}
For any $k=1,\dots,n$ and $t\in T$ the set of essential monomials in $V_t(\omega_k)$ is of the form
\[
E_{a_1,b_1}\dots E_{a_L,b_L},\ 1\le b_1<\dots<b_L\le k<a_L<a_{L-1}<\dots <a_1.
\]  
\end{lemma}
\begin{proof}
Direct computation.
\end{proof}

For a dominant integral $\mu$ let $S(\mu)$ be the Minkowski sum $m_1S_t(\om_1)+\dots +m_nS_t(\om_n)$. 
\begin{corollary}
Let $\mu=\sum_{k=1}^n m_k\om_k$. Then the vectors ${\bE}^\bs v_\mu$, $\bs\in S_t(\mu)$ are linearly independent in
$V_t(\mu)$.
\end{corollary}
\begin{proof}
We prove this by induction on $m_1+\dots+m_n$. If the sum is equal to one, then we are done. Now by definition 
$V_t(\mu+\om_k)=V_t(\mu)\odot V_t(\omega_k)$, where for two cyclic $\fa_t$-modules $U$ and $W$ with cyclic vectors
$u\in U$ and $w\in W$ the module  $U\odot W\subset U\T W$ is the Cartan component ${\mathrm U}(\fa_t) (u\T w)$.
Now one shows that the products of essential monomials for $U$ and $W$ are linearly independent in  $U\odot W$.     
\end{proof}

\begin{corollary}\label{est}
$\dim V_t(\mu)\ge \dim V(\mu)$.
\end{corollary}

\subsection{Lie groups and quiver Grassmannians}
Let ${\rm Gr}_\mathbf{e}(M_t)$ be the quiver Grassmannian corresponding to the representation $M_t$.
To simplify the notation, we assume below that $\mathbf{e}=(1,2,\cdots,n)$. However, all the results of this section hold
in full generality. 

Let $\cO_j$ be the following  line bundles on ${\rm Gr}_j(V)$ generating the Picard group:
$\cO_j=\imath^*\cO(1)$, where $\imath: {\rm Gr}_j(V)\mapsto \bP(\Lambda^jV)$ is the Pl\"ucker embedding.
Then for each $\mu=m_1\om_1+\dots+m_n\om_n$ we obtain the line bundle 
$$\cO(\mu)=\bigotimes_{j=1}^n \cO_j^{\otimes m_j}.$$ 

In a similar way we obtain the line bundle $\cO_t(\mu)$ on each quiver Grassmannian ${\rm Gr}_\be(M_t)$.

\begin{proposition}\label{flat}
For any $t\in T$ we have 
\[
\dim \mathrm{H}^k({\rm Gr}_\be(M_t),\cO_t(\mu))=\delta_{k,0}\dim V(\mu).
\] 
\end{proposition}
\begin{proof}
This follows from the semicontinuity of the dimensions of the cohomology groups in a flat family and 
the known result for $t=0$ in \cite{FF} (the PBW-degenerate flag varieties). 
\end{proof}

For convenience, we extend the parameters $\la_{i,j}$, $2\le i\le j\le n$ to $\la_{i,j}$ with arbitrary $i,j\in\{1,\dots,n+1\}$
by $\la_{1,1}=\la_{n+1,n+1}=1$ and $\la_{i,j}=0$ for other (not yet covered) pairs $i,j$.     
\begin{lemma}
If $\la_{b,a}=0$, then the endomorphisms ${\rm Id}+x\Phi_t(E_{a,b})$, $x\in\bC$ form a group $G_{a,b}$ isomorphic
to the additive group $\bG_a=\bC_+$. If  $\la_{b,a}\ne 0$, then the operators 
${\rm Id}+x\Phi_t(E_{a,b})$, $x\in\bC\setminus \{(a-b-1)\la_{b,a}^{-1}\}$  form a group $G_{a,b}$ isomorphic
to the multiplicative group $\bG_m=\bC^*$.
\end{lemma}
\begin{proof}
We note that 
\[
({\rm Id}+x\Phi_t(E_{a,b}))({\rm Id}+y\Phi_t(E_{a,b}))={\rm Id}+(x+y+xy(b-a+1)\la_{b,a})\Phi_t(E_{a,b}).
\]
This implies the lemma.
\end{proof}

We denote by $G_t$ the group generated by all $G_{a,b}$ and by $G_t^-$ the subgroup generated by $G_{a,b}$ with $a>b$. 
\begin{remark}
The Lie algebra of $G_t$ is isomorphic to $\fa_t$.
\end{remark}

\begin{lemma}
The group $G_t$ acts on the quiver Grassmannian ${\rm Gr}_\mathbf{e}(M_t)$ with an open dense $G_t^-$-orbit through
the point $(\mathrm{span}(v_1,\dots,v_k))_{k=1,\dots,n}$.
\end{lemma}
\begin{proof}
One sees that the $G_t^-$-orbit above has dimension $n(n+1)/2$. Since the quiver Grassmannian  ${\rm Gr}_\bd(M_t)$ is
irreducible, our lemma holds.
\end{proof}

\begin{proposition}
For a regular $\mu$ (i.e. $m_k>0$ for all $k$) 
there exists a natural projective embedding $\imath_\mu:{\rm Gr}_\mathbf{e}(M_t)\subset \bP(V_t(\mu))$.
We have $\imath_\mu^*\cO(1)\simeq \cO_t(\mu)$.
\end{proposition}
\begin{proof}
We have the embedding ${\rm Gr}_\mathbf{e}(M_t)\subset \prod_{k=1}^n \mathrm{Gr}_k((M_t)_k)$, where the left hand side is the closure 
of the $G_t^-$ orbit through the point $\prod_{k=1}^n \mathrm{span}(v_1,\dots,v_k)$. We also have natural $G_t$-equivariant embeddings
$\mathrm{Gr}_k((M_t)_k)\subset \bP(V_t(\omega_k))$. Since $V_\mu(t)$ is the Cartan component inside the tensor product of fundamental 
representations, we obtain the embedding $\imath_\mu:{\rm Gr}_\mathbf{e}(M_t)\subset \bP(V_t(\mu))$.
\end{proof}

\begin{lemma}\label{emb}
There exists an embedding $V_t(\mu)^*\hk \mathrm{H}^0({\rm Gr}_\mathbf{e}(M_t),\cO_t(\mu))$.
\end{lemma}
\begin{proof}
Recall the isomorphism $V_\mu(t)^*= \mathrm{H}^0(\bP(V_\mu(t)),\cO(1))$. Using the embedding $\imath_\mu$ we
consider the restriction map  
$$V_t(\mu)^*=\mathrm{H}^0(\bP(V_t(\mu)),\cO(1))\to \mathrm{H}^0({\rm Gr}_\mathbf{e}(M_t),\imath^*_\mu\cO(1))=
\mathrm{H}^0({\rm Gr}_\mathbf{e}(M_t),\cO_t(\mu)).$$
We claim that this map has no kernel. Indeed, if a section $s$ from $\mathrm{H}^0(\bP(V_t(\mu)),\cO(1))$ vanishes
on the quiver Grassmannian, in particular it vanishes on the open orbit of the group $G_t^-$.
However, the linear span of the vectors from this orbit coincides with the whole $V_t(\mu)$.
Hence, $s\in V_t(\mu)^*$ vanishes on $V_t(\mu)$.  
\end{proof}

\begin{theorem}
$\mathrm{H}^0({\rm Gr}_\mathbf{e}(M_t),\cO_t(\mu))^*\simeq V_t(\mu)$ as $\fa_t$-modules.
\end{theorem}
\begin{proof}
Lemma \ref{emb} gives the surjection from the left hand side to the right hand side. Now   
Proposition \ref{flat} and Corollary \ref{est} imply the Theorem.
\end{proof}

\begin{corollary}
$\dim V_t(\mu)$ is equal to the dimension of the irreducible $\msl_{n+1}$-module of highest weight $\mu$.
\end{corollary}


\begin{thebibliography}{99}
\bibitem{AdF} S.~Abeasis, A.~Del~Fra. \emph{Degenerations for the representations of an equi-oriented quiver of type $A_m$}. Analisi Funzionale e Applicazioni. Suppl. B.U.M.I-Vol.~2 (1980).

\bibitem{ASS} I.~Assem, D.~Simson, A.~Skowronski. \emph{Elements of the representation theory of associative
algebras. Vol. 1. Techniques of representation theory.} London Mathematical Society Student
Texts, 65. Cambridge University Press, Cambridge, 2006.

\bibitem{Bo} K. Bongartz. \emph{Minimal singularities for representations of Dynkin quivers}. Comment. Math. Helv. \textbf{69} (1994), no. 4, 575--611.

\bibitem{Bo2} K. Bongartz. \emph{On degenerations and extensions of finite dimensional modules}. Adv. Math. \textbf{121} (1996), 245--287.



\bibitem{CFFFR} G.~Cerulli Irelli, X.~Fang, E.~Feigin, G.~Fourier, M.~Reineke. \emph{Linear degenerations of flag varieties}. Math. Z. \textbf{287} (2017), no. 1--2, 615--654.

\bibitem{CFR1} G. Cerulli Irelli, E. Feigin, M. Reineke. \emph{Quiver Grassmannians and degenerate flag varieties}. Algebra and Number Theory \textbf{6} (2012), no. 1, 165--194.

\bibitem{CFR2} G. Cerulli Irelli, E. Feigin, M. Reineke. \emph{Desingularization of quiver Grassmannians for Dynkin quivers}, Advances in Mathematics, 2013, no. 245, pp. 182--207.

\bibitem{CFRSchubert} G. Cerulli Irelli, E. Feigin, M. Reineke. \emph{Schubert quiver Grassmannians}. Algebras and Representation Theory, February 2017, Volume 20, Issue 1, pp 147--161.

\bibitem{CL} G. Cerulli Irelli, M. Lanini. \emph{Degenerate flag varieties of type A and C are Schubert varieties}. Internat. Math. Res. Notices, Volume 2015, Issue 15, 1 January 2015, Pages 6353--6374.


\bibitem{Feigin1} E.~Feigin. \emph{Degenerate flag varieties and the median Genocchi numbers}. Math. Res. Lett. \textbf{18} (2011), no. 6, 1163--1178.

\bibitem{Feigin2} E.~Feigin. \newblock {\em  $\mathbb{G}_a^M$ degeneration of flag varieties.} \newblock Selecta Math. (N.S.), \textbf{18}(3):513--537, 2012.

\bibitem{FF} E. Feigin, M. Finkelberg. \emph{Degenerate flag varieties of type A: Forbenius splitting and BW theorem}. Math. Z. 275 (2013), no. 1--2, 55--77.



\bibitem{FeFoLit} E.~Feigin, G.~Fourier, P.~Littelmann. \emph{PBW filtration and bases for irreducible modules in type ${\tt A}_n$}. Transform. Groups 16 (2011), no. 1, 71--89.

\bibitem{FFFM} X.~Fang, E.~Feigin, G.~Fourier, I.~Makhlin, \emph{Weighted PBW degenerations and tropical flag varieties}, to appear in Communications in Contemporary Mathematics, Vol. 21, No. 01, 1850016 (2019).

\bibitem{Fou1} G.~Fourier. \emph{PBW-degenerated Demazure modules and Schubert varieties for triangular elements}. Journal of Combinatorial Theory, Series A, Volume 139, April 2016, Pages 132--152.

\bibitem{G} P.~Gabriel, \emph{Unzerlegbare Darstellungen. I}, Manuscripta Mathematica (1972), 6: 71--103.



\bibitem{LW}
O. Lorscheid, T. Weist, \emph{Pl\"ucker Relations for Quiver Grassmannians}, Algebras and Representation Theory (2018), https://doi.org/10.1007/s10468-017-9762-4.


\bibitem{R_Mon} M.~Reineke. \emph{Monomials in canonical bases of quantum groups and quadratic forms}. J. Pure Appl. Algebra 157 (2001), no. 2--3, 301--309.

\bibitem{R} C.~M.~Ringel. \emph{The Catalan combinatorics of the hereditary Artin algebras}. Preprint 2015. arXiv:1502.06553.

\bibitem{Rior} J. Riordan, \emph{A budget of rhyme scheme counts.} Second International Conference on Combinatorial Mathematics 
(New York, 1978), pp. 455--465, Ann. New York Acad. Sci., 319, New York Acad. Sci., New York, 1979.

\end{thebibliography}
\end{document}